\documentclass[a4paper,12pt,titlepage]{article}
\usepackage{amsmath}
\usepackage{amssymb}
\usepackage[T1]{fontenc}
\usepackage[english]{babel}
\usepackage{makeidx}
\DeclareMathOperator{\diam}{diam}
\DeclareMathOperator{\Int}{Int}
\DeclareMathOperator{\dist}{dist}

\newcommand{\RR}{\mathbb{R}}
\newcommand{\NN}{\mathbb{N}}

\newtheorem{theorem}{Theorem}[section]
\newtheorem{lemma}{Lemma}[section]
\newtheorem{corollary}{Corollary}[section]

\newenvironment{proof}{\par\noindent {\bf Proof.}}
{\begin{flushright} \vspace*{-6mm}\mbox{$\Box$} \end{flushright}}

\newenvironment{proof2}{\par\noindent {\bf Proof of Theorem 5.1.}}
{\begin{flushright} \vspace*{-6mm}\mbox{$\Box$} \end{flushright}}

\newenvironment{proof3}{\par\noindent {\bf Proof of Theorem 6.1.}}
{\begin{flushright} \vspace*{-6mm}\mbox{$\Box$} \end{flushright}}

\newenvironment{proof4}{\par\noindent {\bf Proof of Theorem 7.1 part 1.}}
{\begin{flushright} \vspace*{-6mm}\mbox{$\Box$} \end{flushright}}

\makeindex

\begin{document}

\begin{titlepage}
\vspace*{\fill}
\begin{center}
\begin{picture}(300,510)
  \put( 00,400){\makebox(0,0)[l]{\Huge  \bf \textsc{Harmonic Hardy spaces}}}
  \put( 30,370){\makebox(0,0)[l]{\Huge  \bf \textsc{on smooth domains}}}
  \put(100,250){\makebox(0,0)[l]{\large     \textsc{Tomasz Luks}}}
  \put(90,-80){\makebox(0,0)[bl]{\large \bf \textsc{Wroclaw 2008}}}
\end{picture}
\end{center}
\vspace*{\fill}
\end{titlepage}

\pagestyle{empty}

\tableofcontents

\newpage

\pagestyle{myheadings}

\section{Introduction}
\markright{Introduction \ \& \ Preliminaries}

\medskip
The harmonic Hardy spaces $h^p$ are certain classes of harmonic functions, usually defined on the unit ball or the upper half-space. Many facts from $h^p$ theory have their source in complex analysis and holomorphic Hardy 
spaces $H^p$. They are named in honor of the mathematician G. H. Hardy, who first studied them.

The objective of this paper is to characterize $h^p$ classes and a boundary behavior of harmonic functions 
on a smooth domain in real Euclidean space. Most of the presented results come from the Elias M. Stein's book \cite{S1}. We will concentrate on supplementing the missing or incomplete proofs; the basis will be the well-known 
theory of $h^p$ spaces and nontangential convergence on the ball or the upper half-space.

\section{Preliminaries}

\medskip
Throughout this paper, we will deal with harmonic functions, defined on open subsets of real Euclidean space $\RR^N$, where $N$ will denote a fixed positive integer greater than 1. Let $\Omega$ be an open subset of $\RR^N$; 
a twice continuously differentiable, complex-valued function $u$ is harmonic on $\Omega$ if
\[
	\Delta u=\frac{\partial^2u}{\partial x^2_1}+...+\frac{\partial^2u}{\partial x^2_N}=0
\]
at every point of $\Omega$. The operator $\Delta$ is called the $Laplacian$, and the equation $\Delta u=0$ is called $Laplace's$ $equation$. We will use some well-known properties of harmonic functions, as the maximum principle or the mean value property, 
without a comment; all this properties may be found in \cite{ABR}. 

$x=(x_1,...,x_N)$ denote a typical point in $\RR^N$, and $|x|=(x^2_1+...+x^2_N)^{1/2}$ is the Euclidean 
norm of $x$. By $\langle\cdot,\cdot\rangle$ we denote the usual Euclidean inner product. Recall, that for every 
$x,y\in\RR^N$  
\[
	|x+y|^2=|x|^2+2\langle x,y\rangle+|y|^2.
\]

All functions in this paper are assumed to be complex valued unless stated otherwise. 
For fixed positive integer $k$, a function $f$ is of class $C^k$ on $\Omega$, if $f$ is $k$ times continuously 
differentiable on $\Omega$; $f$ is of class $C^{\infty}$ on $\Omega$, if $f$ is of class $C^k$ for every $k$. We say, 
that $f$ is of class $C^{1,1}$ on $\Omega$, if $f$ is of class $C^1$, and $\nabla f$ satisfies the Lipschitz condition
\[
	|\nabla f(x)-\nabla f(y)|\leq A|x-y|,\quad\forall x,y\in\Omega,
\]
where $A<\infty$ is a positive constant, and
\[
	\nabla f=\left(\frac{\partial f}{\partial x_1},...,\frac{\partial f}{\partial x_N}\right),
\]
is the gradient of $f$. For fixed positive integers $n,k$, a vector-valued function \\ $F\colon\Omega\to\RR^n$, 
$F=(F_1,...,F_n)$, is of class $C^k$ on $\Omega$, if the component functions $F_i$, $i=1,...,n$, are of class $C^k$ on $\Omega$. 

For $E\subset\RR^N$, $C(E)$ denote the space of continuous functions on $E$. $B(a,r)=\left\{x\in\RR^N:|x-a|<r\right\}$ is 
the open ball centered at $a\in\RR^N$ and radius $r>0$. If the dimension is important, we write $B_N(a,r)$ 
in place of $B(a,r)$. The unit sphere, the boundary of $B(0,1)$, is denoted by $S$.

We will also deal with smooth domains in $\RR^N$; by a smooth domain, we mean a bounded domain 
with the boundary at least of class $C^2$. More precisely, we say that a bounded domain $D\subset\RR^N$ has a boundary of class 
$C^2$, if there exists a real-valued function $\lambda$ defined in a neighborhood of $\overline{D}$ with the following properties:

\medskip

\begin{enumerate}
	\item $\lambda$ is of class $C^{2}$.
	\item $\lambda(x)<0$ if and only if $x\in D$.
	\item $\left\{x:\lambda(x)=0\right\}=\partial D$.
	\item $|\nabla\lambda(x)|>0$ if $x\in\partial D$.
\end{enumerate}
\medskip
Throughout this paper, we will assume that $D$ is a bounded domain in $\RR^N$ with the boundary of class $C^{2}$, 
and a function $\lambda$ of the above type will be called a $characterizing$ $function$ for $D$. 
\medskip

\section{Properties of smooth domains in $\RR^N$}
\markright{Properties of smooth domains in $\RR^N$}

\bigskip
 
In this section we will show some important properties of $D$. 
Let $\lambda$ be a characterizing function for $D$. For $y\in\partial D$ denote 
\[
	\nu_y=\frac{\nabla\lambda(y)}{|\nabla\lambda(y)|},
\]
the outward unit normal vector field to $\partial D$. By the property 1 of $\lambda$ (and mean value theorem), there 
exists a positive constant $c$, such that 
\[
	|\nabla\lambda(x)-\nabla\lambda(y)|\leq c|x-y|
\]	
for every $x,y\in\partial D$; property 4 implies, that there exists positive constant $c'$, 
such that for each $y\in\partial D$
\[
	|\nabla\lambda(y)|\geq c'.
\]
Denoting $c_0=2c/c'$, we conclude

\[
	|\nu_x-\nu_y|=\left|\frac{\nabla\lambda(x)}{|\nabla\lambda(x)|}-\frac{\nabla\lambda(y)}{|\nabla\lambda(y)|}\right|=
	\left|\frac{\nabla\lambda(x)}{|\nabla\lambda(x)|}-\frac{\nabla\lambda(y)}{|\nabla\lambda(x)|}+
	\frac{\nabla\lambda(y)}{|\nabla\lambda(x)|}-\frac{\nabla\lambda(y)}{|\nabla\lambda(y)|}\right|
\]
\\
\[
	\leq\frac{|\nabla\lambda(x)-\nabla\lambda(y)|}{|\nabla\lambda(x)|}+
	|\nabla\lambda(y)|\left|\frac{1}{|\nabla\lambda(x)|}-\frac{1}{|\nabla\lambda(y)|}\right|
\]
\\
\[
	=\frac{|\nabla\lambda(x)-\nabla\lambda(y)|}{|\nabla\lambda(x)|}+
	|\nabla\lambda(y)|\left|\frac{|\nabla\lambda(y)|-|\nabla\lambda(x)|}{|\nabla\lambda(x)||\nabla\lambda(y)|}\right|
\]
\\
\[	
	\leq2\frac{|\nabla\lambda(x)-\nabla\lambda(y)|}{|\nabla\lambda(x)|}\leq c_0|x-y|,
\]
\\
for every $x,y\in\partial D$.

For $y\in\partial D$  and $r>0$ let $K(y,r)=B(y,r)\cap(\partial D)$.

\begin{lemma}
There exist positive constants $\rho$, $M_1$, $M_2$, so that to each point $y\in\partial D$ there corresponds a local 
coordinate system $(\overline{x},z)$, where ${\overline{x}\in\RR^{N-1}}$, ${z\in\RR}$ and every point $x\in K(y,\rho)$ is 
represented as $x=(\overline{x},z)$, and a $C^2$ function $\varphi_{y}\colon B_{N-1}(\overline{y},\rho)\to\RR$, such that
\[
	\left|\frac{\partial\varphi_{y}}{\partial x_i}\right|\leq M_1,\quad
	\left|\frac{\partial^2\varphi_{y}}{\partial x_i\partial x_j}\right|\leq M_2,\quad\forall i,j\in\left\{1,...,N-1\right\},
\]
and
\[
	K(y,\rho)=\left\{(\overline{x},\varphi_{y}(\overline{x})):\overline{x}\in
	B_{N-1}(\overline{y},\rho)\right\}\cap B_N(y,\rho).
\]

\end{lemma}

\begin{proof}
Since $\lambda$ is of class $C^2$ in a neighborhood of $\overline{D}$, there exists $c_1,c_2>0$ such that 
for every $y\in\partial D$ and $i,j\in\left\{1,...,N\right\}$
\[
	\left|\frac{\partial\lambda}{\partial x_i}(y)\right|\leq c_1,\quad
	\left|\frac{\partial^2\lambda}{\partial x_i\partial x_j}(y)\right|\leq c_2.
\]
\\
Moreover, because $\partial D$ is compact and $|\nabla\lambda(x)|\geq c$ on $\partial D$ for some $c>0$, 
there exist $r,c'>0$, such that for every $y\in\partial D$ we can choose $i\in\left\{1,...,N\right\}$ so that 
for $x\in\partial D$ and $|x-y|<r$ we have 
\[
	\left|\frac{\partial\lambda}{\partial x_i}(x)\right|\geq c'.
\]
\\ 
So let $y\in\partial D$ and take $i\in\left\{1,...,N\right\}$ like above; for $x\in K(y,r)$ let 
\[
	\overline{x}=(x_1,...,x_{i-1},x_{i+1},...,x_N),\quad z=x_i,
\]
and denote
\[
	(\overline{x},z)=(x_1,...,x_{i-1},z,x_{i+1},...,x_N).
\]
\\
By implicit function theorem, there exist constants $\rho,\rho'$, $0<\rho<\rho'$, and a $C^1$ function 
$\varphi_y\colon B_{N-1}(\overline{y},\rho')\to\RR$, such that:
\begin{enumerate}
	\item $	K(y,\rho)=\left\{(\overline{x},\varphi_{y}(\overline{x})):\overline{x}\in
	B_{N-1}(\overline{y},\rho')\right\}\cap B_N(y,\rho)$.
	\item For every $\overline{x}\in B_{N-1}(\overline{y},\rho')$, $\lambda(\overline{x},\varphi_y(\overline{x}))=0$ and
	\[
		\frac{\partial\varphi_y}{\partial x_i}(\overline{x})=
		-\frac{\frac{\partial\lambda}{\partial x_i}(\overline{x},\varphi_y(\overline{x}))}
		{\frac{\partial\lambda}{\partial z}(\overline{x},\varphi_y(\overline{x}))},\quad i=1,...,N-1.
	\]
\end{enumerate}

Since $\partial D$ is compact, we may choose $\rho,\rho'$ independently on $y$, and take $\rho'=\rho$. 
Additionally we may assume, that $\rho<r$. Hence
\[
	\left|\frac{\partial\varphi_y}{\partial x_i}(\overline{x})\right|\leq c_1/c'=M_1.	
\]
\\
Moreover, by 2. we conclude, that $\varphi_y$ is of class $C^2$ and
\[
	\frac{\partial^2\varphi_y}{\partial x_j\partial x_i}(\overline{x})=
	\frac{\partial}{\partial x_j}\left(-\frac{\frac{\partial\lambda}{\partial x_i}(\overline{x},\varphi_y(\overline{x}))}
	{\frac{\partial\lambda}{\partial z}(\overline{x},\varphi_y(\overline{x}))}\right)
\]

\[
	=\frac{
	\frac{\partial}{\partial x_j}\left(\frac{\partial\lambda}{\partial z}(\overline{x},\varphi_y(\overline{x}))\right)
	\frac{\partial\lambda}{\partial x_i}(\overline{x},\varphi_y(\overline{x}))
	-\frac{\partial}{\partial x_j}\left(\frac{\partial\lambda}{\partial x_i}(\overline{x},\varphi_y(\overline{x}))\right)
	\frac{\partial\lambda}{\partial z}(\overline{x},\varphi_y(\overline{x}))}
	{\left(\frac{\partial\lambda}{\partial z}(\overline{x},\varphi_y(\overline{x}))\right)^2}
\]

\[
	=\frac{\left(\frac{\partial^2\lambda}{\partial x_j\partial z}(\overline{x},\varphi_y(\overline{x}))+
	\frac{\partial^2\lambda}{\partial z^2}(\overline{x},\varphi_y(\overline{x}))
	\frac{\partial\varphi_y}{\partial x_j}(\overline{x})\right)
	\frac{\partial\lambda}{\partial x_i}(\overline{x},\varphi_y(\overline{x}))}
	{\left(\frac{\partial\lambda}{\partial z}(\overline{x},\varphi_y(\overline{x}))\right)^2}
\]

\[
	-\frac{\left(\frac{\partial^2\lambda}{\partial x_j\partial x_i}(\overline{x},\varphi_y(\overline{x}))+
	\frac{\partial^2\lambda}{\partial z\partial x_i}(\overline{x},\varphi_y(\overline{x}))
	\frac{\partial\varphi_y}{\partial x_j}(\overline{x})\right)
	\frac{\partial\lambda}{\partial z}(\overline{x},\varphi_y(\overline{x}))}
	{\left(\frac{\partial\lambda}{\partial z}(\overline{x},\varphi_y(\overline{x}))\right)^2}.
\]
\\
Therefore
\[
	\left|\frac{\partial^2\varphi_y}{\partial x_j\partial x_i}(\overline{x})\right|\leq 
	2\frac{c_1c_2}{c'}\left(1+\frac{c_1}{c'}\right)=M_2.
\]

\end{proof}

\begin{lemma}(ball condition)\\
There exists $r>0$, such that for each $y\in \partial D$ there are balls $B(c_y,r)$, $B(\tilde{c}_y,r)$ that satisfy \medskip
\begin{enumerate}
	\item $\overline{B}(c_y,r)\cap\overline{D^c}=\left\{y\right\}$.
	\item $\overline{B}(\tilde{c}_y,r)\cap\overline{D}=\left\{y\right\}$.
\end{enumerate}

\end{lemma}

\begin{proof}
By Lemma 3.1, there exist positive constants $\rho$ and $M$, such that to each point $y\in\partial D$ there corresponds a local 
coordinate system $(\overline{x},z)$, where $\overline{x}\in\RR^{N-1}$, $z\in\RR$ and every point $x\in K(y,\rho)$ is 
represented as $x=(\overline{x},z)$, and a $C^2$ function $\varphi_{y}\colon B_{N-1}(\overline{y},\rho)\to\RR$, such that 
\[
	\left|\frac{\partial^2\varphi_{y}}{\partial x_i\partial x_j}\right|\leq M,\quad\forall i,j\in\left\{1,...,N-1\right\},
\]
and
\[
	K(y,\rho)=\left\{(\overline{x},\varphi_y(\overline{x})):\overline{x}\in
	B_{N-1}(\overline{y},\rho)\right\}\cap B_N(y,\rho).
\]
\\
By the mean value theorem, for some $M'>0$, every $y\in\partial D$ and every $x\in K(y,\rho)$ we have
\[
	|\nabla\varphi_{y}(\overline{x})-\nabla\varphi_{y}(\overline{x}')|\leq M'|\overline{x}-\overline{x}'|.
\]
We show, that there exists positive constant $r=r(M',\rho)$, such that for every $y\in\partial D$, $B(y-r\nu_{y},r)\subset D$ and $B(y+r\nu_{y},r)\subset \RR^N\backslash D$.  Let $y\in\partial D$; without loss of generality we may assume that $y=0$ and 
$z=x_N$ (in the local coordinate system near 0). Moreover, we may assume (by a rotation of the coordinate system if necessary), 
that $\nu_0=e_N=(0,...,0,1)$. This implies, that $\nabla\varphi_0(0)=0$; therefore 
$|\nabla\varphi_0(\overline{x})|\leq M'|\overline{x}|$ if $|\overline{x}|<\rho$. Moreover, the Taylor expansion gives 
$\varphi_0(\overline{x})=\langle\nabla\varphi_0(\theta\overline{x}),\overline{x}\rangle$ for some $\theta\in(0,1)$, and hence

\[
	|\varphi_0(\overline{x})|\leq|\nabla \varphi_0(\theta\overline{x})|\cdot|\overline{x}|\leq 
	M'\theta|\overline{x}|^2\leq M'|\overline{x}|^2.
\]
\\
Now observe, that $S(re_N,r)$, the sphere with center in $(0,...,0,r)$ and radius $r$, touches the hyperplane 
$\left\{x_N=0\right\}$ at the origin and the lower hemisphere is represented as 

\[
	x_N=\psi_r(\overline{x})=r-\sqrt{r^2-|\overline{x}|^2}=
	r\left(1-\sqrt{1-\left(\frac{|\overline{x}|}{r}\right)^2}\right).
\]
\\
Since $\sqrt{1-t}\leq1-t/2$ for $0\leq t\leq 1$, it follows that for $|\overline{x}|<r$

\[
	\psi_r(\overline{x})\geq\frac{|\overline{x}|^2}{2r}.
\]
\\
So take $r<\min\left\{1/2M',\rho/2\right\}$; for $|\overline{x}|<r$ we have  

\[
	-\psi_r(\overline{x})\leq\varphi_0(\overline{x})\leq\psi_r(\overline{x}),
\]
\\
where $-\psi_r$ represents the upper hemisphere of $S(-re_N,r)$. Because $r<\rho/2$, we conclude, that 

\[
	B(-re_N,r)\subset D,\quad B(re_N,r)\subset\RR^N\backslash D.
\]

\end{proof}

\noindent Obviously, in the lemma above we take $c_y=y-r\nu_y$ and $\tilde{c}_y=y+r\nu_y$. 

Observe, that Lemma 3.2 may be proved without the assertion, that the functions $\varphi_y$ from 
Lemma 3.1 are of class $C^2$ with uniformly bounded derivatives. In fact, it suffices to assume, that for every $y\in\partial D$, 
$\nabla\varphi_y$ satisfies the Lipschitz condition with a constant, which does not depend on $y$. This property characterize the 
domains with the boundary of class $C^{1,1}$. More precisely, we say that a bounded domain $\Omega\subset\RR^N$ has a boundary of class $C^{1,1}$, if there exist positive constants $\rho$, $A$, such that to each point $y\in\partial\Omega$ there 
corresponds a local coordinate system $(\overline{x},z)$, where $\overline{x}\in\RR^{N-1}$, $z\in\RR$ and every point 
${x\in(\partial\Omega)\cap B_N(y,\rho)}$ is represented as $x=(\overline{x},z)$, and a $C^1$ function 
$\varphi_{y}\colon B_{N-1}(\overline{y},\rho)\to\RR$, such that
\[
	|\nabla\varphi_{y}(\overline{x})-\nabla\varphi_{y}(\overline{x}')|\leq A|\overline{x}-\overline{x}'|
\]
and
\[
	(\partial\Omega)\cap B_N(y,\rho)=\left\{(\overline{x},\varphi_y(\overline{x})):\overline{x}\in
	B_{N-1}(\overline{y},\rho)\right\}\cap B_N(y,\rho).
\]
\\
Moreover, it is possible to prove the reverse assertion to Lemma 3.2: if a bounded domain $\Omega\subset\RR^N$ satisfies the ball 
condition, then $\partial\Omega$ is of class $C^{1,1}$ (for the proof, see \cite{AKSZ}).

In the next part of this paper, by $\sigma$ we will denote the area measure on $\partial D$. For $y\in\partial D$, let 
$\rho>0$ and $\varphi_y$ be as in Lemma 3.1. Then, by definition of $\sigma$, for an open set $E\subset K(y,\rho)$ 
we have 
\[
	\sigma(E)=\int_{\left\{\overline{x}:x\in E\right\}}\sqrt{1+|\nabla\varphi(\overline{x})|^2}d\overline{x},
\]
where $\overline{x}$ means the projection of $x$ into $\RR^{k-1}\times\left\{0\right\}\times\RR^{N-k}$, for some 
\\$k\in\left\{1,...,N\right\}$ (for more details, see \cite{SP}).

\begin{lemma}
There exist positive constants $c_1$,$c_2$, such that for each $r>0$,\\ $r\leq\diam(D)$, and each $y\in\partial D$ 
we have:   

\[
	c_1r^{N-1}\leq \sigma\left\{K(y,r)\right\}\leq c_2r^{N-1}.
\]
\end{lemma}

\begin{proof}
By Lemma 3.1, there exist positive constants $\rho$ and $M$, such that to each point $y\in\partial D$ there corresponds a local 
coordinate system $(\overline{x},z)$, where $\overline{x}\in\RR^{N-1}$, $z\in\RR$ and every point $x\in K(y,\rho)$ is 
represented as $x=(\overline{x},z)$, and a $C^2$ function $\varphi_{y}\colon B_{N-1}(\overline{y},\rho)\to\RR$, such that 
\[
	\left|\frac{\partial\varphi_{y}}{\partial x_i}\right|\leq M,\quad i=1,...,N-1,
\]
and
\[
	K(y,\rho)=\left\{(\overline{x},\varphi_{y}(\overline{x})):\overline{x}\in B_{N-1}(\overline{y},\rho)\right\}
	\cap B_N(y,\rho).
\]
\\
By first inequality,
\[
	|\varphi_{y}(\overline{x})-\varphi_{y}(\overline{x}')|\leq \sup|\nabla\varphi_y|\cdot|\overline{x}-\overline{x}'|
	\leq \widetilde{M}|\overline{x}-\overline{x}'|.
\]
\\
Let $r\leq\rho$. If $x\in\left\{(\overline{x},\varphi_{y}(\overline{x})):
\overline{x}\in B_{N-1}\left(\overline{y},r/\sqrt{1+\widetilde{M}^2}\right)\right\}$, then

\[
	|x-y|=\left|(\overline{x}-\overline{y},\varphi_{y}(\overline{x})-\varphi_{y}(\overline{y}))\right|=
	\sqrt{|\overline{x}-\overline{y}|^2+|\varphi_{y}(\overline{x})-\varphi_{y}(\overline{y})|^2}
\]

\[
	\leq\sqrt{1+\widetilde{M}^2}|\overline{x}-\overline{y}|<r,
\]
which means, that $x\in K(y,r)$, and hence

\[
	\sigma\left\{K(y,r)\right\}\geq\sigma\left(\left\{(\overline{x},\varphi_{y}(\overline{x})):
	\overline{x}\in B_{N-1}\left(\overline{y},r/\sqrt{1+\widetilde{M}^2}\right)\right\}\right)
\]

\[
	=\int_{B_{N-1}\left(\overline{y},r/\sqrt{1+\widetilde{M}^2}\right)}\sqrt{1+|\nabla\varphi_y(\overline{x})|^2}d\overline{x}
\]

\[
	\geq m_{N-1}\left\{B_{N-1}\left(\overline{y},r/\sqrt{1+\widetilde{M}^2}\right)\right\}=c\cdot r^{N-1},
\]
\\
where $m_{N-1}$ denotes the $N-1$ dimensional Lebesque measure. Obviously \\ 
$K(y,r)\subset\left\{(\overline{x},\varphi_{y}(\overline{x})):\overline{x}\in B_{N-1}(\overline{y},r)\right\}$. 
Therefore

\[
	\sigma\left\{K(y,r)\right\}\leq
	\int_{B_{N-1}(\overline{y},r)}\sqrt{1+|\nabla\varphi_y(\overline{x})|^2}d\overline{x}
\]

\[	
	\leq\sqrt{1+\widetilde{M}^2}m_{N-1}(B_{N-1}(\overline{y},r))=c'\cdot r^{N-1}.
\]
\\
So the estimate is proved for $r\leq\rho$. Since $\partial D$ is compact, we conclude that $\sigma(\partial D)<\infty$. 
Now suppose $r>\rho$. We have
\[
	\sigma\left\{K(y,r)\right\}\geq \sigma\left\{K(y,\rho)\right\}\geq c\cdot\rho^{N-1}=
	c\cdot\left(\frac{\rho}{r}\right)^{N-1}r^{N-1}
\]

\[
	\geq c\cdot\left(\frac{\rho}{\diam(D)}\right)^{N-1}r^{N-1}=c_1r^{N-1},
\]

\[
	\sigma\left\{K(y,r)\right\}\leq\sigma(\partial D)\leq\frac{\sigma(\partial D)}{\rho^{N-1}}r^{N-1}
	\leq c_2r^{N-1},
\]
\\
where $c_2=\max\left\{c',\sigma(\partial D)/\rho^{N-1}\right\}$.
\\
\end{proof}

Denote $\delta(x)=\dist(x,\partial D)$, the distance of $x$ to the boundary of $D$. Observe, that 
for every $x,x'\in\RR^N$ we have 

\[
	\delta(x)\leq |x-x'|+\delta(x').
\]
By symmetry
\[
	\delta(x')\leq |x-x'|+\delta(x),
\]
and thus
\[
	\delta(x')-|x-x'|\leq \delta(x)\leq |x-x'|+\delta(x')
\]

\[
	-|x-x'|\leq \delta(x)-\delta(x')\leq |x-x'| 
\]

\[
	|\delta(x)-\delta(x')|\leq |x-x'|.
\]
\\
In particular, $\delta$ is continuous function on $\RR^N$. 
For $r>0$, denote 
\[
	D_{r}=\left\{x\in\overline{D}:\delta(x)\leq r\right\}.
\]

\begin{lemma}
Let $r$ be the constant from Lemma 3.2, and let $r_0=r/4$. There exists a map $\pi\colon D_r\to\partial D$ such that 
for every $x\in D_r$ we have
\[
	|\pi(x)-x|=\delta(x)
\]
and 
\[
	|\pi(x)-\pi(y)|\leq 4|x-y|
\]
for every $x,y\in D_{r_0}$.

\end{lemma}

\begin{proof}
By Lemma 3.2, there exists a unique map $\pi$ on $D_r$, that satisfies 
\[
	|\pi(x)-x|=\delta(x).
\] 
In fact, since $\partial D$ is compact, for every $x\in D_r$ there exists $x'\in\partial D$, such that $|x-x'|=\delta(x)$. 
Because $0<\delta(x)\leq r$, $B(x,\delta(x))\cap\partial D=\left\{x'\right\}$ by Lemma 3.2, and we set $\pi(x)=x'$. 

Now choose $x,y\in D_{r_0}$ and suppose $|x-y|<\delta(x)$. Observe, that $\pi(x)=\pi(z)$ for each $z$ from the set
\[
	I=\left\{\pi(x)+t\frac{x-\pi(x)}{\delta(x)}:t\in[0,2\delta(x)]\right\}.
\]
Because $I\subset\overline{B}(x,\delta(x))$, there exists $x'\in I$, such that $|y-x'|=\dist(y,I)$ and 
$\langle y-x',x'-\pi(x)\rangle=0$. Obviously $\delta(x')=|x'-\pi(x)|<2r_0$; by Lemma 3.2
\[
	B(\pi(y)-2(\pi(y)-y),2\delta(y))\subset D
\]
and
\[
	B(\pi(x)-2(\pi(x)-x'),2\delta(x'))\subset D.
\]
Hence we have 
\\
\begin{enumerate}
	\item $2|\pi(y)-y|\leq |\pi(y)-2(\pi(y)-y)-\pi(x)|$
	\\
	\item $2|\pi(x)-x'|-|y-\pi(x)|\leq  |\pi(y)-y| $.
	\\
\end{enumerate}
From 1 we conclude
	
\[
	4|\pi(y)-y|^2\leq |y-\pi(y)+y-\pi(x)|^2
\]

\[
	4|\pi(y)-y|^2\leq |y-\pi(y)|^2+2\langle y-\pi(y),y-\pi(x)\rangle+|y-\pi(x)|^2
\]

\[
	|\pi(y)-y|^2-2\langle y-\pi(y),y-\pi(x)\rangle\leq |y-\pi(x)|^2-2|y-\pi(y)|^2
\]

\[
	|\pi(y)-y|^2-2\langle y-\pi(y),y-\pi(x)\rangle+|y-\pi(x)|^2\leq 2|y-\pi(x)|^2-2|y-\pi(y)|^2
\]

\[
	|\pi(x)-\pi(y)|^2\leq 2|y-\pi(x)|^2-2|y-\pi(y)|^2.
\]
2 gives
\[
	|\pi(x)-\pi(y)|^2\leq 2|y-\pi(x)|^2-2(2|\pi(x)-x'|-|y-\pi(x)|)^2
\]

\[
	=8\delta(x')(|y-\pi(x)|-\delta(x'))=\frac{8\delta(x')(|y-\pi(x)|^2-\delta(x')^2)}{|y-\pi(x)|+\delta(x')}.
\]
\\
Because $\langle y-x',x'-\pi(x)\rangle=0$, we have $|y-\pi(x)|^2=|x'-y|^2+\delta(x')^2$, and hence 

\[
	\frac{8\delta(x')(|y-\pi(x)|^2-\delta(x')^2)}{|y-\pi(x)|+\delta(x')}\leq
	\frac{8\delta(x')(|y-\pi(x)|^2-\delta(x')^2)}{2\delta(x')}
\]

\[
	=4(|x'-y|^2+\delta(x')^2-\delta(x')^2)=4|x'-y|^2.
\]
\\
Obviously $|x'-y|\leq |x-y|$, and thus $|\pi(x)-\pi(y)|\leq 2|x-y|$.

Now if $|x-y|\geq\delta(x)$, then we have
 
\[
	|\pi(x)-\pi(y)|\leq |\pi(x)-x|+|x-y|+|y-\pi(y)|=\delta(x)+|x-y|+\delta(y)
\]

\[
	=2\delta(x)+|x-y|+\delta(y)-\delta(x)\leq 3|x-y|+|\delta(y)-\delta(x)|
\]

\[
	\leq 4|x-y|,
\]
\\
what gives the conclusion of the lemma.

\end{proof}

The map $\pi$ from Lemma 3.4 will be called an $orthogonal$ $projection$. 
Observe, that if $r$ is the constant from Lemma 3.2, then for $x\in D_r$ we have 
\[
	x=\pi(x)-\delta(x)\nu_{\pi(x)}. 
\]

\begin{lemma}
Let $r$ be the constant from Lemma 3.2, and let $r_0=r/4$. Then the function $\delta$ is of class $C^{1,1}$ inside $D_{r_0}$. 
Moreover, for $a\in S$ we have

\[
	\lim_{h\rightarrow 0}\frac{\delta(x+ha)-\delta(x)}{h}=\langle a,-\nu_{\pi(x)}\rangle.
\]
 
\end{lemma}

\begin{proof}
Choose $x\in \Int(D_{r_0})$ and $a\in S$. First observe, that

\[
	2\delta(x)-\left|ha-(x-\pi(x))\right|\leq\delta(x+ha)\leq\left|ha+x-\pi(x)\right|
\]
\\
for $|h|<\delta(x)$. We have

\[
	\delta(x+ha)-\delta(x)\geq\delta(x)-\left|ha-(x-\pi(x))\right|
\]

\[
	=\delta(x)-\sqrt{h^2-2h\langle a,x-\pi(x)\rangle+(\delta(x))^2}
\]

\[
	=h\frac{2\langle a,x-\pi(x)\rangle-h}{\delta(x)+\sqrt{h^2-2h\langle a,x-\pi(x)\rangle+(\delta(x))^2}};
\]
\\
\\
\[
	\delta(x+ha)-\delta(x)\leq\left|ha+x-\pi(x)\right|-\delta(x)
\]

\[
	=\sqrt{h^2+2h\langle a,x-\pi(x)\rangle+(\delta(x))^2}-\delta(x)
\]

\[
	=h\frac{2\langle a,x-\pi(x)\rangle+h}{\sqrt{h^2+2h\langle a,x-\pi(x)\rangle+(\delta(x))^2}+\delta(x)}.
\]
\\
If $h>0$, then

\[
	\frac{\delta(x+ha)-\delta(x)}{h}\geq
\]

\[
	\frac{2\langle a,x-\pi(x)\rangle-h}{\delta(x)+\sqrt{h^2-2h\langle a,x-\pi(x)\rangle+(\delta(x))^2}}
	\stackrel{h\rightarrow 0}{\longrightarrow}\frac{\langle a,x-\pi(x)\rangle}{\delta(x)},
\]
\\
similarly

\[
	\frac{\delta(x+ha)-\delta(x)}{h}\leq
\]

\[
	\frac{2\langle a,x-\pi(x)\rangle+h}{\delta(x)+\sqrt{h^2+2h\langle a,x-\pi(x)\rangle+(\delta(x))^2}}
	\stackrel{h\rightarrow 0}{\longrightarrow}\frac{\langle a,x-\pi(x)\rangle}{\delta(x)};
\]
\\
if $h<0$ the inequalities are reverse. Because $|x-\pi(x)|=\delta(x)$, we have 
\[
	\frac{\pi(x)-x}{\delta(x)}=\nu_{\pi(x)}.
\]
\\
Now $|\pi(x)-\pi(y)|\leq 4|x-y|$ on $D_{r_0}$ by Lemma 3.4; recalling, that $|\nu_z-\nu_{z'}|\leq c_0|z-z'|$ 
on $\partial D$, we conclude
	
\[
	\left|\frac{\partial\delta}{\partial x_i}(x)-\frac{\partial\delta}{\partial x_i}(y)\right|
	=\left|\langle e_i,\nu_{\pi(y)}-\nu_{\pi(x)}\rangle \right|\leq |\nu_{\pi(y)}-\nu_{\pi(x)}|
\]

\[
	\leq c_0|\pi(y)-\pi(x)|\leq 4c_0|x-y|
\]
\\
for every $x,y\in \Int(D_{r_0})$.

\end{proof}

\begin{lemma}
Let $\pi$ be the orthogonal projection. There exists $r>0$, such that $\pi$ is of class $C^1$ inside $D_r$.
\end{lemma}

\begin{proof}
Choose $y\in\partial D$. It suffices to show, that there exists $r>0$, such that $\pi$ is of class $C^1$ on $D\cap B(y,r)$. 
Let $\rho$, $M_1$, $M_2$ be the constants from Lemma 3.1. Assume additionally, that $\rho$ satisfies the assertion of Lemma 3.2. 
Then there exists a local coordinate system $(\overline{x},z)$, where $\overline{x}\in\RR^{N-1}$, $z\in\RR$ 
and every point $x\in K(y,\rho)$ is represented as $x=(\overline{x},z)$, and a $C^2$ function 
$\varphi_{y}\colon B_{N-1}(\overline{y},\rho)\to\RR$, such that
\[
	\left|\frac{\partial\varphi_{y}}{\partial x_i}\right|\leq M_1,\quad
	\left|\frac{\partial^2\varphi_{y}}{\partial x_i\partial x_j}\right|\leq M_2,\quad\forall i,j\in\left\{1,...,N-1\right\},
\]
and
\[
	K(y,\rho)=\left\{(\overline{x},\varphi_{y}(\overline{x})):\overline{x}\in
	B_{N-1}(\overline{y},\rho)\right\}\cap B_N(y,\rho).
\]
\\
Without loss of generality we may assume, that $z=x_N$. Additionally we may assume, that 

\[
	B_N(y,\rho)\cap D=\left\{(\overline{x},x_N):\overline{x}\in
	B_{N-1}(\overline{y},\rho)\wedge x_N<\varphi_{y}(\overline{x})\right\}\cap B_N(y,\rho).
\]
\\
Then for $\overline{x}\in B_{N-1}(\overline{y},\rho)$, $w(\overline{x})=(-\nabla\varphi_y(\overline{x}),1)$ 
is the outward orthogonal vector field to $\partial D$ in $x=(\overline{x},\varphi_{y}(\overline{x}))\in K(y,\rho)$. 
Since $|w(\overline{x})|>0$, we have 
\[
	\frac{w(\overline{x})}{|w(\overline{x})|}=\nu_x.
\]
Let
 
\[
	F(\overline{x},t)=(\overline{x},\varphi_y(\overline{x}))-t\cdot w(\overline{x}),
	\quad\overline{x}\in B_{N-1}(\overline{y},\rho),t\in\RR.
\]
\\
Because $\varphi_y$ is of class $C^2$, $F$ is of class $C^1$ on $B_{N-1}(\overline{y},\rho)\times\RR$. 

Observe, that since $\rho$ satisfies the condition of Lemma 3.2, for $\rho'\leq\rho/2$ and $x\in D\cap B_N(y,\rho')$ we have

\[
	|\pi(x)-y|\leq |\pi(x)-x|+|x-y|=\delta(x)+|x-y|\leq 2|x-y|<\rho,
\]
\\
so $\pi(x)\in K(y,\rho)$. Therefore, if we denote $\pi(x)=\left(\overline{\pi(x)},\varphi_y\left(\overline{\pi(x)}\right)\right)$, 
then 

\[
	F\left(\overline{\pi(x)},\frac{\delta(x)}{|w(\overline{x})|}\right)=
	\left(\overline{\pi(x)},\varphi_y\left(\overline{\pi(x)}\right)\right)-\delta(x)\frac{w(\overline{x})}{|w(\overline{x})|}
\]

\[
	=\pi(x)-\delta(x)\nu_{\pi(x)}=x.
\]
\\
Since $|w(\overline{x})|\geq1$, we have

\[
	D\cap B_N(y,\rho')\subset F(B_{N-1}(\overline{y},\rho)\times (0,\rho')).
\]
\\
The Jacobian matrix of $F(\overline{x},t)=(\overline{x}+t\nabla\varphi_y(\overline{x}),\varphi_y(\overline{x})-t)$ has a form

\[
	J(\overline{x},t)=
	\left[
	\begin{matrix}
		1+t\frac{\partial^2g}{\partial x_1^2}(\overline{x}) & t\frac{\partial^2g}{\partial x_2\partial x_1}(\overline{x}) & \ldots 
		& t\frac{\partial^2g}{\partial x_{N-1}\partial x_1}(\overline{x}) & \frac{\partial g}{\partial x_1}(\overline{x}) \\
		t\frac{\partial^2g}{\partial x_1\partial x_2}(\overline{x}) & 1+t\frac{\partial^2g}{\partial x_2^2}(\overline{x}) & \ldots
		& t\frac{\partial^2g}{\partial x_{N-1}\partial x_2}(\overline{x}) & \frac{\partial g}{\partial x_2}(\overline{x}) \\
		\vdots & \vdots & \ddots & \vdots & \vdots \\ t\frac{\partial^2g}{\partial x_1\partial x_{N-1}}(\overline{x}) & 
		t\frac{\partial^2g}{\partial x_2\partial x_{N-1}}(\overline{x}) & \ldots & 1+t\frac{\partial^2g}{\partial x_{N-1}^2}(\overline{x})
		& \frac{\partial g}{\partial x_{N-1}}(\overline{x}) \\ \frac{\partial g}{\partial x_1}(\overline{x}) & 
		\frac{\partial g}{\partial x_2}(\overline{x}) & \ldots & \frac{\partial g}{\partial x_{N-1}}(\overline{x}) & -1
	\end{matrix}
	\right].
\]
\\
Therefore, if $t\rightarrow0$, then $J$ tends to the matrix 

\[
	J_0(\overline{x})=
	\left[
	\begin{matrix}
		1 & 0 & \ldots & 0 & \frac{\partial g}{\partial x_1}(\overline{x}) \\ 
		0 & 1 & \ldots & 0 & \frac{\partial g}{\partial x_2}(\overline{x}) \\
		\vdots & \vdots & \ddots & \vdots & \vdots \\ 0 & 0 & \ldots & 1 & \frac{\partial g}{\partial x_{N-1}}(\overline{x}) \\
		\frac{\partial g}{\partial x_1}(\overline{x}) & 
		\frac{\partial g}{\partial x_2}(\overline{x}) & 
		\ldots & \frac{\partial g}{\partial x_{N-1}}(\overline{x}) & -1 
	\end{matrix}
	\right].
\]
\\
Since $\varphi_y$ is of class $C^2$ and 
\[
	\left|\frac{\partial\varphi_{y}}{\partial x_i}\right|\leq M_1,\quad
	\left|\frac{\partial^2\varphi_{y}}{\partial x_i\partial x_j}\right|\leq M_2,\quad\forall i,j\in\left\{1,...,N-1\right\},
\]
\\
$\det J\stackrel{t\rightarrow0}{\longrightarrow}\det J_0$ uniformly on $B_{N-1}(\overline{y},\rho)$. A simple calculation shows, 
that $\det J_0(\overline{x})=-|\nabla\varphi_y(\overline{x})|^2-1$. Hence, there exists $r>0$, such that 
${\det J(\overline{x},t)\neq0}$ for every $x\in B_{N-1}(\overline{y},\rho)$ and $t\in (0,r)$. We may assume, that $r\leq\rho/2$. By 
inverse function theorem, $F$ is invertible in $B_{N-1}(\overline{y},\rho)\times (0,r)$, and $F^{-1}$ is of class $C^1$ in 
$F(B_{N-1}(\overline{y},\rho)\times (0,r))$. In particular, $F^{-1}$ is of class $C^{1}$ in $D\cap B_N(y,r)$. 

Now if $x\in D\cap B_N(y,r)$, then $x=F(\overline{x}',t)$ 
for some $\overline{x}'\in B_{N-1}(\overline{y},\rho)$ and $t\in (0,r)$. Thus, if we denote $F^{-1}(x)=(F_1^{-1}(x),...,F_N^{-1}(x))$, 
then for $i\in\left\{1,...,N-1\right\}$ we have $F_i^{-1}(x)=x_i'$. On the other side, 

\[
	x'=(\overline{x}',\varphi_y(\overline{x}'))=\pi(x)=(\pi_1(x),...,\pi_N(x)), 
\]
\\
since $\rho$ satisfies the condition of Lemma 3.2, and $r<\rho$. Hence 
\[
	\pi_i(x)=x_i'=F_i^{-1}(x),\ i=1,...,N-1.
\]
Moreover, 
\[
	\pi_N(x)=x_N'=\varphi_y(\overline{x}')=\varphi_y(\pi_1(x),...,\pi_{N-1}(x)),
\]
\\
and thus $\pi$ is of class $C^1$ in $D\cap B(y,r)$, as desired.

\end{proof}

\begin{corollary}
Let $r_0$ be as in Lemma 3.5. There exists $r>0$, $r\leq r_0$, such that $\delta$ is of class $C^2$ inside $D_r$. 
\end{corollary}

\begin{proof}
By Lemma 3.5, for every $x\in D_{r_0}$ and $i=1,...,N$ we have

\[
	\frac{\partial\delta}{\partial x_i}(x)=\langle e_i,-\nu_{\pi(x)}\rangle = \frac{\langle e_i,x-\pi(x)\rangle}{\delta(x)}
	=\frac{x_i-\pi_i(x)}{\delta(x)}.
\]
\\
By Lemma 3.6, there exists $r>0$, such that $\pi=(\pi_1,...,\pi_N)$ is of class $C^1$ inside $D_r$. If we assume additionally, 
that $r\leq r_0$, then $\delta$ is of class $C^2$ inside $D_r$.
\\
\end{proof}


Now we will introduce some facts from classical harmonic analysis on a smooth domain $D\subset\RR^N$. 
Let $G_D(x,y)$ be the Green's function for $D$, defined in $(D\times\overline{D})\backslash\left\{(x,x):x\in\ D\right\}$. 
It is uniquely determined by the following properties:
\begin{enumerate}
	\item $G_D$ is of class $C^2$ on $(D\times D)\backslash\left\{(x,x):x\in D\right\}$ and of class $C^{2-\varepsilon}$ up to 
	$(D\times\overline{D})\backslash\left\{(x,x):x\in D\right\}$.
	\item $\Delta_yG_D(x,y)=0$ for every $y\in D$ and $y\neq x$.
	\item $G_D(x,y)+\Gamma_N(x-y)$ is harmonic on $D$ for each fixed $x\in D$, where $\Gamma_N$ is the fundamental solution 
	for the Laplacian on $\RR^N$, given by
	\[
		\Gamma_N(x)=
		\begin{cases}
		(2\pi)^{-1}\log|x|, & \ N=2,        \\
		(2-N)^{-1}\omega^{-1}_{N-1}|x|^{2-N}, & \ N>2,
		\end{cases}
	\]
	$\omega_{N-1}$ is the area measure of $S$ in $\RR^N$.
	\item $G_D(x,y)|_{y\in\partial D}=0$ for each fixed $x\in D$.
\end{enumerate}

The function 
\[
	P_D(x,y)=-\frac{\partial{G_D(x,y)}}{\partial{\nu_y}},
\] 
defined in $D\times\partial D$ is the Poisson kernel for $D$. It is continuous and strictly positive on $D\times\partial D$. 
Moreover, by the use of Green's theorem, one can prove the following property: if $u$ is harmonic on $D$ and continuous on 
$\overline{D}$, then 

\[
	u(x)=\int_{\partial D}P_D(x,y)u(y)d\sigma(y),\quad x\in D
\]
\\
(for more details and proofs, see \cite{K}). 
Now, since $D$ is a domain with the boundary of class $C^2$, the Dirichlet Problem is solvable. More precisely, for each 
$f\in C(\partial D)$, there exists $F\in C(\overline{D})$, such that $F|_{\partial D}=f$ and $F$ is harmonic on $D$ 
(a good source for these results is \cite{ABR}). Hence 

\[
	F(x)=\int_{\partial D}P_D(x,y)f(y)d\sigma(y),\quad x\in D.
\]
\\
Moreover, the maximum principle for harmonic functions implies, that $F$ is unique (since $D$ is bounded). Using this facts, 
one can prove the following known properties of the Poisson kernel for $D$:
\begin{enumerate}
	\item $P_D(x,\cdot)\geq c_x>0$, for each fixed $x\in D$.
	\item For every $x\in D$,
			\[
				\int_{\partial D}P_D(x,y)d\sigma(y)=1. 
			\]
	\item For any $\eta>0$ and any fixed $y\in\partial D$,
			\[
				\lim_{D\ni x\rightarrow y}\int_{\partial D\backslash K(y,\eta)}P_D(x,y)d\sigma(y)=0.
			\]
	\item $P_D(\cdot,y)$ is harmonic on $D$ for each fixed $y\in\partial D$.
\end{enumerate}

Unfortunately, $P_D$ almost never can be computed explicitly. However, in the special case when $D=B(a,r)$, 
the Poisson kernel is given by 

\[
	P_{B(a,r)}(x,y)=\frac{1}{\omega_{N-1}r}\cdot\frac{r^2-|x-a|^2}{|x-y|^N}
\]
\\ 
By the formula above and Lemma 3.2, we have the following inequality (more details also may be found in \cite{K}).

\begin{lemma}
There exists a positive constant $C$, such that for every $x\in D$ and $y\in\partial D$ we have
\[
	P_D(x,y)\leq\frac{C}{|x-y|^{N-1}}.
\]

\end{lemma}

\begin{proof}
Fix $x\in D$, $y\in\partial D$. Choose, by Lemma 2, $r>0$ (which does not depend on $y$), such that 

\[
	\widetilde{B}_y=B(y+r\nu_y,r)\subset D^c.
\]
\\
Let $G_{\widetilde{B}^{c}_{y}}$ be the Green's function for $\widetilde{B}^{c}_{y}$. Observe that $G_D(x,\cdot)=0$ on $\partial D$,  
whereas $G_{\widetilde{B}^{c}_{y}}(x,\cdot)\geq0$ on $\partial D$. Since $G_{\widetilde{B}^{c}_{y}}(x,\cdot)-G_D(x,\cdot)$ 
is harmonic on $D$, it follows that 

\[
	G_D(x,t)\leq G_{\widetilde{B}^{c}_{y}}(x,t),\quad t\in\overline{D}
\]

\[
	G_D(x,y)=G_{\widetilde{B}^{c}_{y}}(x,y)=0.
\]
\\
Therefore
\[
	P_D(x,y)=-\frac{\partial{G_D(x,y)}}{\partial{\nu_y}}\leq 
	-\frac{\partial{G_{\widetilde{B}^{c}_{y}}(x,y)}}{\partial{\nu_y}}=P_{\widetilde{B}^{c}_{y}}(x,y).
\]
\\
The Poisson kernel for $\widetilde{B}_{y}$ has a form

\[
	P_{\widetilde{B}_{y}}(x,t)=\frac{1}{\omega_{N-1}r}\cdot\frac{r^2-|x-\tilde{c}_y|^2}{|x-t|^N},
	\quad x\in \widetilde{B}_{y},t\in\partial \widetilde{B}_{y},
\]
\\
where $\tilde{c}_y=y+r\nu_y$. By Kelvin transform, the Poisson kernel for $\widetilde{B}^{c}_{y}$ is 

\[
	P_{\widetilde{B}^{c}_{y}}(x,t)=\frac{1}{\omega_{N-1}r}\cdot\frac{|x-\tilde{c}_y|^2-r^2}{|x-t|^N}.
\]
\\
Hence we have

\[
	P_D(x,y)\leq\frac{1}{\omega_{N-1}r}\cdot\frac{|x-\tilde{c}_y|^2-r^2}{|x-t|^N}
	=\frac{1}{\omega_{N-1}r}\cdot\frac{(|x-\tilde{c}_y|-|\tilde{c}_y-y|)(|x-\tilde{c}_y|+r)}{|x-t|^N}
\]

\[
	\leq\frac{1}{\omega_{N-1}r}\cdot\frac{(|x-y|)(|x-y|+2r)}{|x-t|^N}\leq\frac{\diam(D)+2r}{\omega_{N-1}r}\cdot\frac{1}{|x-y|^{N-1}}
	=\frac{C}{|x-y|^{N-1}}.
\]
\\
\end{proof}

\medskip
In the next sections we will deal with the Banach spaces $L^p(\partial D)$, where $1\leq p\leq\infty$. When $p\in[1,\infty)$, $L^p(\partial D)$ consists of the Borel measurable functions $f$ on $\partial D$ for which
\[
	\left\|f\right\|_p=\left(\int_{\partial D}|f|^pd\sigma\right)^{\frac{1}{p}}<\infty;
\]
$L^{\infty}(\partial D)$ consists of the Borel measurable functions $f$ on $\partial D$ for which 
$\left\|f\right\|_{\infty}<\infty$, where $\left\|f\right\|_{\infty}$ denotes the essential supremum norm on $\partial D$ with 
respect to $\sigma$. The number $q\in[1,\infty]$ is said to be $conjugate$ to $p$ if $1/p+1/q=1$. If $1\leq p<\infty$ and $q$ 
is conjugate to $p$, then $L^q(\partial D)$ is the dual space of $L^p(\partial D)$; we identify $g\in L^q(\partial D)$ with 
the linear functional $\Lambda_g$ on $L^p(\partial D)$ defined by
\[
	\Lambda_g(f)=\int_{\partial D}fg\ d\sigma.
\]

Let $1\leq p<\infty$ and let $q$ be conjugate to $p$; we say, that the sequence $\left\{g_n\right\}\subset L^q(\partial D)$ 
converges weak$^{\ast}$ to $g\in L^q(\partial D)$, if $\Lambda_{g_n}(f)\stackrel{n}{\rightarrow}\Lambda_g(f)$ for every 
$f\in L^p(\partial D)$. Note that because $\sigma$ is a finite measure on $\partial D$, ${L^p(\partial D)\subset L^1(\partial D)}$ 
for all $p\in[1,\infty]$. Recall also that $C(\partial D)$ is dense in $L^{p}(\partial D)$ for $1\leq p<\infty$. 

For $f\in L^{1}(\partial D)$ and $\mu\in M(\partial D)$ (the set of complex Borel measures on $\partial D$), 
we define the Poisson integrals of $f$ and $\mu$, respectively as

\[
	P_D[f](x)=\int_{\partial D}{P_D(x,y)f(y)d\sigma(y)},
\]

\[
	P_D[\mu](x)=\int_{\partial D}{P_D(x,y)d\mu(y)}.
\]
\\
Differentiating under the integral sign, we see that $P_D[f]$ and $P_D[\mu]$ are harmonic for every 
$f\in L^{1}(\partial D)$ and $\mu\in M(\partial D)$.
\medskip

\section{The spaces $\textit{h}^{\textit{p}}\left(\textit{D}\right)$}\medskip
\markright{The spaces $\textit{h}^{\textit{p}}\left(\textit{D}\right)$}

\bigskip

As we said before, $D$ is a bounded domain with the boundary of class $C^2$, and $\lambda$ is a characterizing 
function for $D$. Of course there are infinitely many such characterizing functions. Now we will show, that every 
characterizing function is comparable to $\delta$ near $\partial D$. Recall that 
$D_r=\left\{x\in\overline{D}:\delta(x)\leq r\right\}$.

\begin{lemma}
Let $\lambda$ be a characterizing function for $D$. There exist positive constants $C,r$, such that for every $x\in D_r$ 
\[
	\frac{\delta(x)}{C}\leq-\lambda(x)\leq C\delta(x).
\]\end{lemma}

\begin{proof}
By the properties of characterizing function, there exist positive constants $c_1,c_2,r_1$, such that 
$c_1\leq|\nabla\lambda(x)|\leq c_2$ for every $x\in D_{r_1}$. Assume additionally, that $r_1$ satisfies the condition of Lemma 3.2. Moreover, we may assume, that for some $c_3>0$ and every $x,y\in D_{r_1}$ we have
\[
	|\nabla\lambda(x)-\nabla\lambda(y)|\leq c_3|x-y|,
\]
since $\lambda$ is of class $C^2$ in a neighborhood of $\overline{D}$. Hence, similarly as in section 3 we conclude
\[
	\left|\frac{\nabla\lambda(x)}{|\nabla\lambda(x)|}-\frac{\nabla\lambda(y)}{|\nabla\lambda(y)|}\right|\leq c_4|x-y|,
	\quad x,y\in D_{r_1},
\]
\\
where $c_4=2c_3/c_1$. Take $r_2=\min\left\{r_1,1/c_4\right\}$, and fix $x\in D_{r_2}$. By Lemmas 3.2 and 3.4, 
$B(x,\delta(x))\cap\partial D=\left\{\pi(x)\right\}$. Moreover, for every $t\in(0,1)$, 
\[
	tx+(1-t)\pi(x)\in D_{r_2}.
\] 
Let 
\[
	f(t)=\lambda(tx+(1-t)\pi(x)),\quad t\in [0,1].
\]
Then $f$ is differentiable, and 
\[
	f'(t)=\langle\nabla\lambda(tx+(1-t)\pi(x)),x-\pi(x)\rangle.
\]
Thus we have
\[
	-\lambda(x)=-(\lambda(x)-\lambda(\pi(x)))=-(f(1)-f(0))=-f'(\theta)
\]

\[
	=\langle\nabla\lambda(\theta x+(1-\theta)\pi(x)),\pi(x)-x\rangle,
\]
\\
for some $\theta\in(0,1)$. Denote $x_{\theta}=\theta x+(1-\theta)\pi(x)$; therefore

\[
	-\lambda(x)=\langle\nabla\lambda(x_{\theta}),\pi(x)-x\rangle
	=\langle\frac{\nabla\lambda(x_{\theta})}{|\nabla\lambda(x_{\theta})|},\frac{\pi(x)-x}{\delta(x)}\rangle
	|\nabla\lambda(x_{\theta})|\delta(x)
\]

\[
	=\langle\frac{\nabla\lambda(x_{\theta})}{|\nabla\lambda(x_{\theta})|},
	\frac{\nabla\lambda(\pi(x))}{|\nabla\lambda(\pi(x))|}\rangle|\nabla\lambda(x_{\theta})|\delta(x)
\]

\[
	=\frac{2-\left|\frac{\nabla\lambda(x_{\theta})}{|\nabla\lambda(x_{\theta})|}-
	\frac{\nabla\lambda(\pi(x))}{|\nabla\lambda(\pi(x))|}\right|^2}{2}|\nabla\lambda(x_{\theta})|\delta(x).
\]
\\
Now observe, that 

\[
	\left|\frac{\nabla\lambda(x_{\theta})}{|\nabla\lambda(x_{\theta})|}-\frac{\nabla\lambda(\pi(x))}{|\nabla\lambda(\pi(x))|}\right|
	\leq c_4|x_{\theta}-\pi(x)|=c_4\theta|x-\pi(x)|\leq c_4\delta(x)\leq c_4r_2\leq 1,
\]
\\
since $x_{\theta}\in D_{r_2}$. Hence 

\[
	\frac{c_1}{2}\delta(x)\leq-\lambda(x)\leq c_2\delta(x),
\]
\\
what gives the conclusion of the lemma.

\end{proof}

Each characterizing function $\lambda$ determines a family of approximating subdomains $D^{\varepsilon}_{\lambda}$, 
for $\varepsilon$ sufficiently small and positive. Clearly, by the properties of characterizing function, there exists $c,r'>0$, such that for every $x\in D_{r'}$, $|\nabla\lambda(x)|\geq c$. Moreover, we can choose 
$\varepsilon_{\lambda}>0$ with the property, that if $0<\varepsilon<\varepsilon_{\lambda}$ and $\lambda(x)=-\varepsilon$, then $x\in D_{r'}$ (Lemma 4.1 may be helpful here). For $\varepsilon$ as above let 
$D^{\varepsilon}_{\lambda}=\left\{x:\lambda(x)<-\varepsilon\right\}$. 
Then $\partial D^{\varepsilon}_{\lambda}$ is the level surface $\left\{x:\lambda(x)=-\varepsilon\right\}$, and 
$\lambda(x)+\varepsilon$ is a characterizing function for $D^{\varepsilon}_{\lambda}$ (thus $\partial D^{\varepsilon}_{\lambda}$ is of class $C^2$).

Let $\sigma_{\varepsilon}$ be the area measure on $\partial D^{\varepsilon}_{\lambda}$. Let $\pi$ be the orthogonal projection onto $\partial D$. Then we may choose $\varepsilon_{\lambda}$ so small, so that for each 
$0<\varepsilon<\varepsilon_{\lambda}$, $\sigma_{\varepsilon}$ is locally a transform of the measure $\sigma$ in the following sense. 
For $y\in\partial D^{\varepsilon}_{\lambda}$ let $(\overline{x},z)$ be the local coordinate system near $y$ and 
$\varphi^{\varepsilon}_y$ a real-valued $C^2$ function as in Lemma 3.1 (thus every point $x\in\partial D^{\varepsilon}_{\lambda}$ 
near $y$ is represented as $x=(\overline{x},z)$, where $z=\varphi^{\varepsilon}_y(\overline{x})$, and $\overline{x}$ 
means the projection of $x$ into $\RR^{k-1}\times\left\{0\right\}\times\RR^{N-k}$, for some $k\in\left\{1,...,N\right\}$). 
We may assume, that the same local coordinate system $(\overline{x},z)$ 
corresponds to the neighborhood of $\pi(y)$ in $\partial D$ (which means the same projection $\overline{x}$, but here 
$z=\varphi_{\pi(y)}(\overline{x})$). For $\overline{x}$ near $\overline{y}$ denote 
\[
\widetilde{\pi}_{\varepsilon}(\overline{x})=\overline{\pi(\overline{x},\varphi^{\varepsilon}_y(\overline{x}))}.
\]
Then there exists $\rho_0>0$ (which does not depend on $\varepsilon$ and $y$, $0<\varepsilon<\varepsilon_{\lambda}$, 
$y\in\partial D^{\varepsilon}_{\lambda}$), such that 
\[
	\det J_{\widetilde{\pi}_{\varepsilon}}(\overline{x})\geq c_{\varepsilon}>0,\quad  \forall
	\overline{x}\in\left\{\overline{x}:x\in B(y,\rho_0)\cap\partial D^{\varepsilon}_{\lambda}\right\},
\]
where $J_{\widetilde{\pi}_{\varepsilon}}$ is the Jacobian matrix of $\widetilde{\pi}_{\varepsilon}$. Moreover, 
$\det J_{\widetilde{\pi}_{\varepsilon}}(\overline{x})$ tends to 1 as $\varepsilon\rightarrow0$ uniformly with respect to $\overline{x}\in\left\{\overline{x}:x\in B(y,\rho_0)\cap\partial D^{\varepsilon}_{\lambda}\right\}$ (what can be proved explicitly, by the use of some facts contained in the proof of Lemma 3.6). Thus  $\widetilde{\pi}_{\varepsilon}$ 
is invertible in $\left\{\overline{x}:x\in B(y,\rho_0)\cap\partial D^{\varepsilon}_{\lambda}\right\}$, and we have 

\[
	\sigma_{\varepsilon}(B(y,\rho_0)\cap\partial D^{\varepsilon}_{\lambda})=
	\int_{\left\{\overline{x}:x\in B(y,\rho_0)\cap\partial D^{\varepsilon}_{\lambda}\right\}}
	\sqrt{1+|\nabla\varphi^{\varepsilon}_y(\overline{x})|^2}d\overline{x}
\]

\[
	=\int_{\widetilde{\pi}_{\varepsilon}\left(\left\{\overline{x}:x\in B(y,\rho_0)\cap\partial D^{\varepsilon}_{\lambda}\right\}\right)}
	\sqrt{1+|\nabla\varphi^{\varepsilon}_y(\widetilde{\pi}^{-1}_{\varepsilon}(\overline{w}))|^2}
	\det J_{\widetilde{\pi}^{-1}_{\varepsilon}}(\overline{w})d\overline{w}
\]

\[
	=\int_{K_{\varepsilon}}\frac{\sqrt{1+|\nabla\varphi^{\varepsilon}_y(\widetilde{\pi}^{-1}_{\varepsilon}(\overline{w}))|^2}}
	{\sqrt{1+|\nabla\varphi_{\pi(y)}(\overline{w})|^2}}\det J_{\widetilde{\pi}^{-1}_{\varepsilon}}(\overline{w})d\sigma(w),
\]
where
\[
	K_{\varepsilon}=\left\{(\overline{w},\varphi_{\pi(y)}(\overline{w})):\overline{w}\in
	\widetilde{\pi}_{\varepsilon}\left(\left\{\overline{x}:x\in B(y,\rho_0)\cap\partial D^{\varepsilon}_{\lambda}\right\}\right)\right\}.
\]
\\
Therefore we conclude, that if $0<\varepsilon<\varepsilon_{\lambda}$, then 
$\pi_{\varepsilon}:=\pi|_{\partial D^{\varepsilon}_{\lambda}}$ is invertible, and for 
$f\in C(\partial D^{\varepsilon}_{\lambda})$ we have 

\[
	\int_{\partial D^{\varepsilon}_{\lambda}}f(x)d\sigma_{\varepsilon}(x)=
	\int_{\partial D}f(\pi^{-1}_{\varepsilon}(y))\phi_{\varepsilon}(y)d\sigma(y),
\]
\\
where $\phi_{\varepsilon}$ is locally well defined, and tends to 1 uniformly as $\varepsilon\rightarrow0$ 
(thus we may assume, that $\phi_{\varepsilon}\leq2$ for every $0<\varepsilon<\varepsilon_{\lambda}$).

Now for fixed characterizing function $\lambda$, $\varepsilon_{\lambda}$ as above and $1\leq p\leq\infty$, 
we define a space $\textit{h}^p(D)$ to be the class of function $u$ harmonic on $D$, for which
 
\[
	\left\|u\right\|^{\lambda}_{\textit{h}^p}=\sup_{0<\varepsilon<\varepsilon_{\lambda}}
	\left(\int_{\partial D^{\varepsilon}_{\lambda}}|u(x)|^pd\sigma_{\varepsilon}(x)\right)^\frac{1}{p}<\infty.
\]
\\
The spaces $\textit{h}^p(D)$ are called "harmonic Hardy spaces". 
Note that $\textit{h}^{\infty}(D)$ is simply the collection of functions harmonic and bounded on $D$, and that 
\[
	\left\|u\right\|^{\lambda}_{\textit{h}^{\infty}}=\left\|u\right\|_{\textit{h}^{\infty}}=\sup_{x\in D}|u(x)|.
\]
Moreover, $\textit{h}^p(D)\subset\textit{h}^q(D)$ for $1\leq q<p\leq\infty$. 

Observe, that since $\sigma_{\varepsilon}(\partial D^{\varepsilon}_{\lambda})\leq2\sigma(\partial D)$ for every 
$0<\varepsilon<\varepsilon_{\lambda}$, the harmonic function $u\in\textit{h}^p(D)$ if and only if

\[
	\sup_{0<\varepsilon<\varepsilon'}
	\left(\int_{\partial D^{\varepsilon}_{\lambda}}|u(x)|^pd\sigma_{\varepsilon}(x)\right)^\frac{1}{p}<\infty
\]
\\
for every $0<\varepsilon'<\varepsilon_{\lambda}$. 

The next lemma shows, that the definition of $\textit{h}^p(D)$ does not depend on characterizing function.

\begin{lemma} (Stein)\\
Let $\lambda_1$ and $\lambda_2$ be two characterizing functions for $D$. Then for each $p$, ${1\leq p\leq\infty}$, and each harmonic function $u$ on $D$, the two conditions 

\[
	\sup_{0<\varepsilon<\varepsilon_{\lambda_i}}
	\left(\int_{\partial D^{\varepsilon}_{\lambda_i}}|u(x)|^pd\sigma^{i}_{\varepsilon}(x)\right)^\frac{1}{p}<\infty,\quad i=1,2,
\]
\\
are equivalent ($\sigma^{i}_{\varepsilon}$ is the area measure on $\partial D^{\varepsilon}_{\lambda_i}$).

\end{lemma}

\begin{proof}
It suffices to show, that the last condition for $i=1$ implies the same condition for $i=2$. Because for $p=\infty$ the conclusion is trivial, we may assume that $1\leq p<\infty$. Denote 
\[
	M=\sup_{0<\varepsilon<\varepsilon_{\lambda_1}}
	\int_{\partial D^{\varepsilon}_{\lambda_1}}|u(x)|^pd\sigma^{1}_{\varepsilon}(x).
\]
\\
Let $r>0$ be the constant from Lemma 4.1, such that $\lambda_1$ and $\lambda_2$ are comparable to $\delta$ in $D_r$. 
Clearly, there exists $C>0$, such that 

\[
	\frac{\delta(x)}{C}\leq-\lambda_i(x)\leq C\delta(x),\quad i=1,2,\quad x\in D_r.
\]
\\
We may assume, that $C\geq 1$. Take $\varepsilon_0>0$, so that if $0<\varepsilon<\varepsilon_0$ and $\lambda_i(x)=-\varepsilon$, $i=1,2$, then $x\in D_r$. Assume additionally, that 
\[
	\varepsilon_0\leq\frac{\min\left\{\varepsilon_{\lambda_1},\varepsilon_{\lambda_2},r\right\}}{2C^2}.
\]
First we show, that there exist positive constants $c,c_1,c_2$, so that if $0<\varepsilon<\varepsilon_0$ 
and $x\in\partial D^{\varepsilon}_{\lambda_2}$, then 

\[
	B(x,c\varepsilon)\subset\left\{y:-c_1\varepsilon<\lambda_1(y)<-c_2\varepsilon\right\}=L_{\varepsilon}.
\]
\\
Let $c<\frac{1}{2C}$, and choose $0<\varepsilon<\varepsilon_0$, $x\in\partial D^{\varepsilon}_{\lambda_2}$. 
Thus, by Lemma 4.1
\[
	\frac{\delta(x)}{C}\leq-\lambda_2(x)\leq C\delta(x),
\]

\[
	\frac{\delta(x)}{C}\leq\varepsilon\leq C\delta(x),
\]

\[
	\frac{\varepsilon}{C}\leq\delta(x)\leq C\varepsilon,
\]
\\
since $x\in D_r$. If $y\in B(x,c\varepsilon)$, then we have 	

\[
	\delta(x)-|x-y|\leq\delta(y)\leq\delta(x)+|x-y|,
\]

\[
	\varepsilon\left(\frac{1}{C}-c\right)<\delta(y)<(C+c)\varepsilon.
\]
\\
Since $(C+c)\varepsilon<2C\varepsilon_0\leq r$, $y\in D_r$ and thus 
\[
 \frac{\delta(y)}{C}<-\lambda_1(y)<C\delta(y),
\]

\[
	\frac{\varepsilon}{C}\left(\frac{1}{C}-c\right)<-\lambda_1(y)< C(C+c)\varepsilon.
\]
\\
Denote $c_1=C(C+c)$, $c_2=\frac{1}{C}\left(\frac{1}{C}-c\right)$. Therefore we have 
\[
	B(x,c\varepsilon)\subset L_{\varepsilon}=\left\{y:-c_1\varepsilon<\lambda_1(y)<-c_2\varepsilon\right\}.
\]
for every $0<\varepsilon<\varepsilon_0$ and $x\in\partial D^{\varepsilon}_{\lambda_2}$.

Now, by the mean value property and Jensen inequality,

\[
	|u(x)|^p=\left|c_3\varepsilon^{-N}\int_{B(x,c\varepsilon)}u(y)dy\right|^p
	\leq c_3\varepsilon^{-N}\int_{B(x,c\varepsilon)}|u(y)|^pdy.
\]
\\
Therefore, for $0<\varepsilon<\varepsilon_0$ we have

\[
	\int_{\partial D^{\varepsilon}_{\lambda_2}}|u(x)|^pd\sigma^{2}_{\varepsilon}(x)\leq
	c_3\varepsilon^{-N}\int_{\partial D^{\varepsilon}_{\lambda_2}}
	\left(\int_{B(x,c\varepsilon)}|u(y)|^pdy\right)d\sigma^{2}_{\varepsilon}(x)
\]

\[
	=c_3\varepsilon^{-N}\int_{\RR^N}\left(\int_{\partial D^{\varepsilon}_{\lambda_2}}
	\chi_{\varepsilon}(x,y)d\sigma^{2}_{\varepsilon}(x)\right)|u(y)|^pdy,
\]
\\
where $\chi_{\varepsilon}(x,y)$ is the characteristic function of the ball $B(x,c\varepsilon)$. Observe, that 
\[
	\int_{\partial D^{\varepsilon}_{\lambda_2}}\chi_{\varepsilon}(x,y)d\sigma^{2}_{\varepsilon}(x)=0
\]
for $y\notin L_{\varepsilon}=\left\{y:-c_1\varepsilon<\lambda_1(y)<-c_2\varepsilon\right\}$. Moreover, there exists 
positive constant $c_4$, such that for every $0<\varepsilon<\varepsilon_0$ and $y\in L_{\varepsilon}$ we have
\[
	\int_{\partial D^{\varepsilon}_{\lambda_2}}\chi_{\varepsilon}(x,y)d\sigma^{2}_{\varepsilon}(x)\leq c_4\varepsilon^{N-1}
\]
(by Lemma 3.3 and the properties of the transform $\phi_{\varepsilon}$). Hence

\[
	\int_{\partial D^{\varepsilon}_{\lambda_2}}|u(x)|^pd\sigma^{2}_{\varepsilon}(x)\leq
	c_3c_4\varepsilon^{-1}\int_{L_{\varepsilon}}|u(y)|^pdy
\]

\[
	=c_5\varepsilon^{-1}\int^{c_1\varepsilon}_{c_2\varepsilon}\left(
	\int_{\partial D^{\eta}_{\lambda_1}}|u(y)|^pd\sigma^{1}_{\eta}(y)\right)d\eta\leq c_5(c_1-c_2)M,
\]
\\
since $c_1\varepsilon<2C^2\varepsilon_0\leq\varepsilon_{\lambda_1}$. Thus the lemma is proved.

\end{proof}

By definition, if $u$ is harmonic and bounded on $D$, then $u\in\textit{h}^{\infty}(D)$. The next lemma shows a similar result for $p<\infty$.

\begin{lemma}
Let $1\leq p<\infty$. Suppose $u$ is harmonic on $D$, and there exists a positive harmonic function $h$ on $D$, such that 
$|u(x)|^p\leq h(x)$, for every $x\in D$. Then $u\in\textit{h}^p(D)$. 

\end{lemma}

\begin{proof}
Fix $x_0\in D$ and let $\lambda_0(x)=-G_D(x_0,x)$. Since 
\[
	|\nabla\lambda_0(y)|=|\nabla_y G_D(x_0,y)|=-\frac{\partial{G_D(x_0,y)}}{\partial{\nu_y}}=P_D(x_0,y)\geq c_{x_0}>0,
\]
for every $y\in\partial D$ and $\lambda_0|_{\partial D}\equiv0$, $\lambda_0$ is a characterizing function for $D$ 
(near the boundary). Let $\varepsilon_0=\varepsilon_{\lambda_0}>0$ be so small, so that 
$D^{\varepsilon}_{0}=D^{\varepsilon}_{\lambda_0}$ are well-defined approximating subdomains for $0<\varepsilon<\varepsilon_0$. 
Then $D^{\varepsilon}_{0}$ are of class $C^2$ and have their Green's functions $G_{D^{\varepsilon}_{0}}$ and Poisson kernels $P_{D^{\varepsilon}_{0}}$.

Now observe, that by the properties of the Green's function, 
\[
	G_{D^{\varepsilon}_{0}}(x_0,x)=G_D(x_0,x)-\varepsilon,
\]
for every $x\in D^{\varepsilon}_{0}$, $x\neq x_0$, since $G_D(x_0,\cdot)|_{\partial D^{\varepsilon}_{0}}\equiv\varepsilon$. Hence 

\[
	P_{D^{\varepsilon}_{0}}(x_0,x)=|\nabla_x G_{D^{\varepsilon}_{0}}(x_0,x)|=|\nabla_x G_D(x_0,x)|,\quad 
	x\in\partial D^{\varepsilon}_{0},
\]
and
\[
	h(x_0)=\int_{\partial D^{\varepsilon}_{0}}P_{D^{\varepsilon}_{0}}(x_0,x)h(x)d\sigma_{\varepsilon}(x)=
	\int_{\partial D^{\varepsilon}_{0}}|\nabla_x G_D(x_0,x)|h(x)d\sigma_{\varepsilon}(x).
\]
\\

Now, because $|\nabla_y G_D(x_0,\pi(x))|=P_D(x_0,\pi(x))\geq c_{x_0}>0$ for every ${x\in{D}}$, where $\pi$ the orthogonal projection, there exists $\varepsilon_1>0$, such that 
\[
	|\nabla_y G_D(x_0,x)|\geq \frac{c_{x_0}}{2},
\]
if $\delta(x)<\varepsilon_1$. Since $\delta|_{\partial D^{\varepsilon}_{0}}$ tends to 0 uniformly as $\varepsilon\rightarrow0$, 
we may choose $\varepsilon_2$, $0<\varepsilon_2<\varepsilon_0$, such that the last inequality holds for 
$x\in\partial D^{\varepsilon}_{0}$, whenever $0<\varepsilon<\varepsilon_2$. Thus

\[
	h(x_0)=\int_{\partial D^{\varepsilon}_{0}}|\nabla_x G_D(x_0,x)|h(x)d\sigma_{\varepsilon}(x)\geq 
	\frac{c_{x_0}}{2}\int_{\partial D^{\varepsilon}_{0}}h(x)d\sigma_{\varepsilon}(x),
\]
\\
for $0<\varepsilon<\varepsilon_2$, and we have

\[
	\sup_{0<\varepsilon<\varepsilon_2}\int_{\partial D^{\varepsilon}_{0}}|u(x)|^pd\sigma_{\varepsilon}(x)\leq
	\sup_{0<\varepsilon<\varepsilon_2}\int_{\partial D^{\varepsilon}_{0}}h(x)d\sigma_{\varepsilon}(x)\leq 
	\frac{2h(x_0)}{c_{x_0}}.
\]
\\
By Lemma 4.2 we conclude, that $u\in\textit{h}^p(D)$.

\end{proof}

A simple corollary of Lemma 4.3 is, that every positive harmonic function on $D$ belongs to $\textit{h}^1(D)$. 
In particular, for fixed $y\in\partial D$, $P_D(\cdot,y)\in\textit{h}^1(D)$. Moreover we have
\begin{enumerate}
	\item If $\mu\in M(\partial D)$, then $P_D[\mu]\in\textit{h}^1(D)$.
	\item If $1\leq p\leq\infty$ and $f\in L^p(\partial D)$, then $P_D[f]\in\textit{h}^p(D)$.
\end{enumerate}

To see 1, choose $\mu\in M(\partial D)$ and let $u=P_D[\mu]$. Then for every $x\in D$ we have 
\[
	|u(x)|=\left|\int_{\partial D}P_D(x,y)d\mu(y)\right|\leq\int_{\partial D}P_D(x,y)d|\mu|(y).
\]
Since $|\mu|$ is positive and finite measure on $\partial D$, $P_D[|\mu|]$ is positive and harmonic in $D$, 
and by Lemma 4.3, $u\in\textit{h}^1(D)$. Proof of 2 is similar. For fixed $1\leq p<\infty$, $f\in L^p(\partial D)$, 
let $u=P_D[f]$. Then by Jensen inequality we have 
\[
	|u(x)|^p=\left|\int_{\partial D}P_D(x,y)f(y)d\sigma(y)\right|^p\leq\int_{\partial D}P_D(x,y)|f(y)|^pd\sigma(y)
	=P_D\left[|f|^p\right](x),
\]
for every $x\in D$. Since $P_D\left[|f|^p\right]$ is positive, $u\in\textit{h}^p(D)$. The case $p=\infty$ is the easiest. For $f\in L^{\infty}(\partial D)$ and $u=P_D[f]$ we have 
\[
	|u(x)|\leq \int_{\partial D}P_D(x,y)|f(y)|d\sigma(y)\leq\left\|f\right\|_{\infty}\int_{\partial D}P_D(x,y)d\sigma(y)=\left\|f\right\|_{\infty},
\]
and thus $u\in\textit{h}^{\infty}(D)$.

In the next part of $\textit{h}^p$ theory, we will need some stronger assertion about the functions
$\left\{P_D(\cdot,y)\right\}_{y\in\partial D}$.

\begin{lemma}
Let $\lambda$ be a characterizing function for $D$ and choose $\varepsilon_{\lambda}>0$ as before. There exists a positive constant $C_{\lambda}$, such that for every $y\in\partial D$
\[
	\left\|P_D(\cdot,y)\right\|^{\lambda}_{\textit{h}^1}\leq C_{\lambda}.
\]
\end{lemma}

\begin{proof}
Fix $x_0\in D$ and let $\lambda_0(x)=-G_D(x_0,x)$. Then, as in the proof of Lemma 4.3 we conclude, that 
$\lambda_0$ is a characterizing function for $D$. Moreover, since $P_D(\cdot,y)$ are positive and harmonic on $D$ for every $y\in\partial D$, there exists $\varepsilon_1>0$ and  $M_{x_0}>0$ (which does not depend on $y$), such that 

\[
\sup_{0<\varepsilon<\varepsilon_1}\int_{\partial D^{\varepsilon}_{\lambda_0}}P_D(x,y)d\sigma_{\varepsilon}(x)\leq 
M_{x_0}P_D(x_0,y).
\]
\\
By Lemma 3.7, for some $C>0$ and every $x\in D$, $y\in\partial D$,
\[
	P_D(x,y)\leq\frac{C}{|x-y|^{N-1}}\leq\frac{C}{\delta(x)^{N-1}}.
\]
Therefore

\[
	\left\|P_D(\cdot,y)\right\|^{\lambda_0}_{\textit{h}^1}=
	\sup_{0<\varepsilon<\varepsilon_{\lambda_0}}\int_{\partial D^{\varepsilon}_{\lambda_0}}P_D(x,y)d\sigma_{\varepsilon}(x)
\]

\[
	\leq\sup_{0<\varepsilon<\varepsilon_1}\int_{\partial D^{\varepsilon}_{\lambda_0}}P_D(x,y)d\sigma_{\varepsilon}(x)+
	\sup_{\varepsilon_1\leq\varepsilon<\varepsilon_{\lambda_0}}
	\int_{\partial D^{\varepsilon}_{\lambda_0}}P_D(x,y)d\sigma_{\varepsilon}(x)
\]

\[
	\leq \frac{CM_{x_0}}{\delta(x_0)^{N-1}}+
	\frac{2C\sigma(\partial D)}{\dist(D^{\varepsilon_1}_{\lambda_0},\partial D)}=C_{\lambda_0}.
\]
\\

Now for any characterizing function $\lambda$, $\varepsilon_{\lambda}$ as before, we may choose, as in the proof 
of Lemma 4.2, $\varepsilon_2>0$, $\varepsilon_2<\min\left\{\varepsilon_{\lambda_0},\varepsilon_{\lambda}\right\}$, such that

\[
	\sup_{0<\varepsilon<\varepsilon_2}\int_{\partial D^{\varepsilon}_{\lambda}}P_D(x,y)d\sigma_{\varepsilon}(x)\leq
	M_1\left\|P_D(\cdot,y)\right\|^{\lambda_0}_{\textit{h}^1},
\]
\\
for some $M_1=M_1(\lambda,\lambda_0)>0$. Because the functions $P_D(\cdot,y)$ are uniformly bounded by some 
$M_2=M_2(\lambda)>0$ on the sets $\partial D^{\varepsilon}_{\lambda}$ for $\varepsilon_2\leq\varepsilon<\varepsilon_{\lambda}$, we have

\[
	\left\|P_D(\cdot,y)\right\|^{\lambda}_{\textit{h}^1}\leq M_1\left\|P_D(\cdot,y)\right\|^{\lambda_0}_{\textit{h}^1}
	+M_22\sigma(\partial D)\leq M_1C_{\lambda_0}+M_22\sigma(\partial D)=C_{\lambda}<\infty,
\]
\\
what gives the conclusion of the lemma.

\end{proof}

Recall, that for fixed characterizing function $\lambda$, $\varepsilon_{\lambda}>0$ as usually, the map
\[
	\pi_{\varepsilon}:=\pi|_{\partial D^{\varepsilon}_{\lambda}},\quad 0<\varepsilon<\varepsilon_{\lambda},
\]
is invertible. For $u$ harmonic on $D$ denote $u_{\varepsilon}(y)=u(\pi^{-1}_{\varepsilon}(y))$, $y\in\partial D$. 
If $f\in C(\partial D)$ and $u=P_D[f]$, then $u_{\varepsilon}\rightarrow f$ in $C(\partial D)$. Because $C(\partial D)$ is dense in $L^p(\partial D)$ for $1\leq p<\infty$, we have the following result on $L^p$-convergence.

\begin{lemma}
Suppose $1\leq p<\infty$. If $f\in L^p(\partial D)$ and $u=P_D[f]$, then 
\[
	\left\|u_{\varepsilon}-f\right\|_p\rightarrow0\quad as\quad \varepsilon\rightarrow0.
\]
\end{lemma}

\begin{proof}
Choose $1\leq p<\infty$, $f\in L^p(\partial D)$ and let $u=P_D[f]$. Fix $\varepsilon>0$. 
Let $C_{\lambda}\geq1$ satisfy the assertion of Lemma 4.4, and choose $g\in C(\partial D)$, such that 
$\left\|f-g\right\|_p<\varepsilon/C_{\lambda}$. Let $v=P_D[g]$, and choose $\varepsilon_1>0$, such that 
$\left\|v_{\varepsilon'}-g\right\|_p<\varepsilon$ for every $0<\varepsilon'<\varepsilon_1$. Then we have 

\[
	\left\|u_{\varepsilon'}-f\right\|_p\leq \left\|u_{\varepsilon'}-v_{\varepsilon'}\right\|_p+
	\left\|v_{\varepsilon'}-g\right\|_p+\left\|f-g\right\|_p\leq\left\|u_{\varepsilon'}-v_{\varepsilon'}\right\|_p 
	+ 2\varepsilon.
\]
However
\[
	\left\|u_{\varepsilon'}-v_{\varepsilon'}\right\|^p_p=\int_{\partial D}
	\left|\int_{\partial D}P_D(\pi^{-1}_{\varepsilon'}(y),z)(f(z)-g(z))d\sigma(z)\right|^pd\sigma(y)
\]

\[
	\leq\int_{\partial D}\int_{\partial D}P_D(\pi^{-1}_{\varepsilon'}(y),z)|f(z)-g(z)|^pd\sigma(z)d\sigma(y),
\]

\[
	=\int_{\partial D}\int_{\partial D}P_D(\pi^{-1}_{\varepsilon'}(y),z)d\sigma(y)|f(z)-g(z)|^pd\sigma(z),
\]
\\
by Jensen inequality and Fubini's theorem. Now we may choose $\varepsilon_2>0$, $\varepsilon_2<\varepsilon_1$, such that 
for every $0<\varepsilon'<\varepsilon_2$, and every $z\in\partial D$ we have

\[
	\int_{\partial D}P_D(\pi^{-1}_{\varepsilon'}(y),z)d\sigma(y)\leq
	2\int_{\partial D}P_D(\pi^{-1}_{\varepsilon'}(y),z)\phi_{\varepsilon'}(y)d\sigma(y)
\]

\[
	=2\int_{\partial D^{\varepsilon'}_{\lambda}}P_D(x,z)d\sigma_{\varepsilon'}(x)\leq 
	2\left\|P_D(\cdot,z)\right\|^{\lambda}_{\textit{h}^1}\leq 2C_{\lambda}.
\]
Therefore
\[
	\left\|u_{\varepsilon'}-v_{\varepsilon'}\right\|_p=\left(\int_{\partial D}\int_{\partial D}
	P_D(\pi^{-1}_{\varepsilon'}(y),z)d\sigma(y)|f(z)-g(z)|^pd\sigma(z)\right)^{\frac{1}{p}}
\]

\[
	\leq (2C_{\lambda})^{\frac{1}{p}}\left\|f-g\right\|_p<2\varepsilon,
\]
and hence $\left\|u_{\varepsilon'}-f\right\|_p< 4\varepsilon$, for every $0<\varepsilon'<\varepsilon_2$. Since 
$\varepsilon$ is arbitrary, we have 
\[
	\left\|u_{\varepsilon}-f\right\|_p\stackrel{\varepsilon\rightarrow0}{\longrightarrow}0,
\] 
as desired.

\end{proof}

As we have seen before, if $1\leq p\leq\infty$ and $f\in L^p(\partial D)$, then ${P_D[f]\in\textit{h}^p(D)}$. 
To the end of this section we show, that for $p>1$ each harmonic function from the space $\textit{h}^p(D)$ can 
be characterized in terms of the Poisson kernel for $D$.

\begin{theorem}
Suppose $1<p\leq\infty$ and let $u\in\textit{h}^p(D)$. Then there exists $f\in L^p(\partial D)$, such that $u=P_D[f]$. Moreover, if $\lambda$ is a characterizing function for $D$, then for some positive constant 
$\widetilde{C}=\widetilde{C}({\lambda},p)$ we have
\[
	\left\|f\right\|_p\leq \left\|u\right\|^{\lambda}_{\textit{h}^p}\leq\widetilde{C}\left\|f\right\|_p.
\]
\end{theorem}

\begin{proof}
Let $\left\{D_j\right\}$ be a finite cover of $D$ with the following properties:
\begin{enumerate}
	\item $D=\bigcup D_j$.
	\item For every $j$, $D_j$ is a domain with the boundary of class $C^2$.
	\item For every $j$, $\partial D_j\cap\partial D$ is a $N-1$ dimensional manifold with boundary.
	\item There exists $\varepsilon_0>0$ and a vector $\nu_j=\nu_{y_j}$ for some $y_j\in\partial D_j\cap\partial D$, so that
	$\overline{D}_j-\varepsilon\nu_j\subset D$ for every $0<\varepsilon<\varepsilon_0$.
\end{enumerate}

Let $P_j(x,y)$ be the Poisson kernel for $D_j$. Because for every $0<\varepsilon<\varepsilon_0$, the functions $u^j_{\varepsilon}(x)=u(x-\varepsilon\nu_j)$ are harmonic on $D_j$ and continuous on $\overline{D}_j$, we have
\[
	u^j_{\varepsilon}(x)=\int_{\partial D_j}P_j(x,y)u^j_{\varepsilon}(y)d\sigma_j(y),\quad x\in D_j.
\]
Moreover, in view of Lemma 4.2,
\[
	\sup_{0<\varepsilon<\varepsilon_0}\int_{\partial D_j}|u^j_{\varepsilon}(y)|^pd\sigma_j(y)<\infty.
\]

By Banach-Alaoglu theorem, there exists a subsequence $u^j_{\varepsilon_k}$, that converges weak$^{\ast}$ to some 
$f_j\in L^p(\partial D_j)$. Hence, if $q$ is conjugate to $p$, then obviously $P_j(x,\cdot)\in L^q(\partial D_j)$, and we have 
\[
	u(x)=\int_{\partial D_j}P_j(x,y)f_j(y)d\sigma_j(y),\quad x\in D_j.
\]

Now observe, that if $\lambda_j,\lambda_k$ are the characterizing functions for the domains $D_j,D_k$ respectively, 
and $\partial D_j\cap\partial D_k$ contains some open subset of $\partial D$, then $\lambda_j,\lambda_k$ are comparable 
on the set 
\[
	\left\{y-t\nu_y:y\in\partial D_j\cap\partial D_k\cap\partial D,0<t<t_0\right\},
\]
for $t_0$ sufficiently small. Hence, by (a small modification of) Lemma 4.5, $f_j=f_k$ a.e. (with respect to $\sigma$) in 
$\partial D_j\cap\partial D_k\cap\partial D$. Thus $f\equiv f_j$ on $\partial D_j\cap\partial D$ is well defined; obviously, 
$f\in L^p(\partial D)$. 

It remains to be shown, that $u=P_D[f]$. So fix $x_0\in D$, and let ${\lambda(x)=G_D(x_0,x)}$. As in the proof of Lemma 4.3 we conclude, that

\[
	P_{D^{\varepsilon}_{\lambda}}(x_0,x)=|\nabla_x G_{D^{\varepsilon}_{\lambda}}(x_0,x)|=|\nabla_x G_D(x_0,x)|,\quad 
	x\in\partial D^{\varepsilon}_{\lambda},
\]
and
\[
	u(x_0)=\int_{\partial D^{\varepsilon}_{\lambda}}P_{D^{\varepsilon}_{\lambda}}(x_0,x)u(x)d\sigma_{\varepsilon}(x)=
	\int_{\partial D^{\varepsilon}_{\lambda}}|\nabla_x G_D(x_0,x)|u(x)d\sigma_{\varepsilon}(x)
\]

\[
	=\int_{\partial D^{\varepsilon}_{\lambda}}|\nabla\lambda(x)|u(x)d\sigma_{\varepsilon}(x).
\]
Moreover,

\[
	\int_{\partial D^{\varepsilon}_{\lambda}}|\nabla\lambda(x)|u(x)d\sigma_{\varepsilon}(x)=
	\int_{\partial D}|\nabla\lambda(\pi^{-1}_{\varepsilon}(y))|u(\pi^{-1}_{\varepsilon}(y))\phi_{\varepsilon}(y)d\sigma(y).
\]
\\
Hence we have

\[
	u(x_0)=\int_{\partial D}|\nabla\lambda(\pi^{-1}_{\varepsilon}(y))|u(\pi^{-1}_{\varepsilon}(y))\phi_{\varepsilon}(y)d\sigma(y)
\]

\[
	=\int_{\partial D}\left(|\nabla\lambda(\pi^{-1}_{\varepsilon}(y))|\phi_{\varepsilon}(y)-
	|\nabla_yG_D(x_0,y)\right)u(\pi^{-1}_{\varepsilon}(y))d\sigma(y)
\]

\[
	+\int_{\partial D}P_D(x_0,y)u(\pi^{-1}_{\varepsilon}(y))d\sigma(y)=I_1+I_2.
\]
\\
Now since $|\nabla\lambda(\pi^{-1}_{\varepsilon}(y))|\rightarrow |\nabla_y G_D(x_0,y)|$, $\phi_{\varepsilon}\rightarrow1$ 
uniformly on $\partial D$ as $\varepsilon\rightarrow0$, and $u\in\textit{h}^p(D)\subset\textit{h}^1(D)$, we have 
$I_1\stackrel{\varepsilon}{\rightarrow}0$, and it suffices to show that $I_2\stackrel{\varepsilon}{\rightarrow} P_D[f](x_0)$. For fixed $j$ we have 

\[
	\left|\int_{\partial D\cap\partial D_j}P_D(x_0,y)u(\pi^{-1}_{\varepsilon}(y))d\sigma(y)-
	\int_{\partial D\cap\partial D_j}P_D(x_0,y)f(y)d\sigma(y)\right|
\]

\[
	\leq\frac{C}{\delta(x_0)^{N-1}}\int_{\partial D\cap\partial D_j}|u(\pi^{-1}_{\varepsilon}(y))-f(y)|d\sigma(y),
\]
\\
by Lemma 3.7. Thus if $\lambda_j$ is a characterizing function for $D_j$, then $\lambda$ and $\lambda_j$ are comparable on the set $\pi^{-1}_{\varepsilon}(\partial D\cap\partial D_j)\subset D_j$ (independently on $\varepsilon$), and by the use of Lemma 4.5 we obtain 
\[
	\int_{\partial D\cap\partial D_j}|u(\pi^{-1}_{\varepsilon}(y))-f(y)|d\sigma(y)
	\stackrel{\varepsilon\rightarrow0}{\longrightarrow}0.
\]
Because $\partial D=\bigcup\partial D_j\cap\partial D$ (and the sum is finite), we have 
$I_2\stackrel{\varepsilon}{\rightarrow} P_D[f](x_0)$. Therefore $u=P_D[f]$. 

Now let $\lambda'$ be a characterizing function for $D$ (which is not necessarily $\lambda$). Suppose first, that 
$p<\infty$. Then
 
\[
	\int_{\partial D^{\varepsilon}_{\lambda'}}|u(x)|^pd\sigma_{\varepsilon}(x)=
	\int_{\partial D}|u(\pi^{-1}_{\varepsilon}(y))|^p\phi_{\varepsilon}(y)d\sigma(y)
	\stackrel{\varepsilon\rightarrow0}{\longrightarrow}\left\|f\right\|^p_p,
\]
\\
by Lemma 4.5. Therefore $\left\|u\right\|^{\lambda}_{\textit{h}^p}\geq\left\|f\right\|_p$. On the other hand, 

\[
	\int_{\partial D^{\varepsilon}_{\lambda'}}|u(x)|^pd\sigma_{\varepsilon}(x)=
	\int_{\partial D^{\varepsilon}_{\lambda'}}\left|\int_{\partial D}P_D(x,y)f(y)d\sigma(y)\right|^pd\sigma_{\varepsilon}(x)
\]

\[
	\leq\int_{\partial D}\int_{\partial D^{\varepsilon}_{\lambda'}}P_D(x,y)d\sigma_{\varepsilon}(x)|f(y)|^pd\sigma(y)
	\leq C_{\lambda'}\left\|f\right\|^p_p,
\]
\\
by Lemma 4.4. Thus $\left\|u\right\|^{\lambda'}_{\textit{h}^p}\leq\left(C_{\lambda'}\right)^{1/p}\left\|f\right\|_p$ 
(obviously, the estimation occurs even if $f\in L^1(\partial D)$). 

If $p=\infty$, then $\left\|u\right\|_{\textit{h}^{\infty}}=\sup|u|=\sup|P_D[f]|\leq\left\|f\right\|_{\infty}$. 
By previous computations, $\left\|u\right\|^{\lambda'}_{\textit{h}^q}\geq\left\|f\right\|_q$ 
for $1\leq q<\infty$. Since $\left\|u\right\|^{\lambda'}_{\textit{h}^q}\rightarrow\left\|u\right\|_{\textit{h}^{\infty}}$ 
as $q\rightarrow\infty$, we then get $\left\|u\right\|_{\textit{h}^{\infty}}=\left\|f\right\|_{\infty}$. That finishes the proof of the theorem.
\\
\end{proof}

\section{The Fatou Theorem}\medskip
\markright{The Fatou Theorem}

\bigskip

In this section, by the use of \cite{K} and \cite{S2}, we will prove the known property of the Poisson integrals on $D$, a so called "nontangential" convergence.

For each $y\in\partial D$ and $\alpha>0$ we define the "cone" of aperture $\alpha$ and vertex $y$ as

\[
	\Gamma_{\alpha}(y)=\left\{x\in D:|x-y|<(1+\alpha)\delta(x)\right\}.
\]
\\
Obviously $\Gamma_{\alpha}(y)\subset\Gamma_{\beta}(y)$ for $\alpha<\beta$, and $\bigcup_{\alpha>0}\Gamma_{\alpha}(y)=D$. 
We say that a function $u$ on $D$ has a $nontangential$ $limit$ $L$ at $y\in\partial D$ if, for every $\alpha>0$, 
$u(x)\rightarrow L$ as $x\rightarrow y$ within $\Gamma_{\alpha}(y)$.

\begin{theorem}
Suppose $u=P_D[f]$, where ${f\in L^{1}(\partial D)}$. Then $u$ has a nontangential limit at almost 
every point of $\partial D$ and
\[
	\lim_{\Gamma_{\alpha}(y)\ni x\to y}u(x)=f(y)\quad for\quad a.e.\quad y\in\partial D.
\]
\end{theorem}

A key tool in the proof of Theorem 5.1 is the use of the Hardy-Littlewood maximal functions. For any 
$f\in L^{1}(\partial D)$ we define 

\[
	M[f](y)=\sup_{r>0}\frac{1}{\sigma(K(y,r))}\int_{K(y,r)}|f(z)|d\sigma(z),
\]
\\
the maximal function of $f$.

\begin{lemma}(Wiener)\\
For positive integer $k$, let $F\subset\RR^k$ be a compact set that is covered by the open balls 
$\left\{B_{\alpha}\right\}_{\alpha\in A}$, $B_{\alpha}=B(c_{\alpha},r_{\alpha})$. There is a disjoint, finite subcover 
$B_{\alpha_1}$, $B_{\alpha_2}$,..., such that 
\[
	F\subset\bigcup_i B(c_{\alpha_i},3r_{\alpha_i}).
\]
\end{lemma}

\begin{proof}
Since $F$ is compact, we may assume that $\left\{B_{\alpha}\right\}_{\alpha\in A}$ is finite. Let $B_{\alpha_1}$ be the ball in this 
collection, that has the largest radius. Let $B_{\alpha_2}$ be the ball that is disjoint from $B_{\alpha_1}$ and has the greatest radius, and 
so on. The process ends in finitely many steps. We claim that the $B_{\alpha_i}$ chosen above satisfy the conclusion of the lemma.

It is enough to show that $B_{\alpha}\subset\bigcup_i B(c_{\alpha_i},3r_{\alpha_i})$ for every $\alpha$. Fix an $\alpha$. If $\alpha=\alpha_i$ for some $i$, then we are done. If $\alpha\notin\left\{\alpha_i\right\}$, let $\alpha_{i_0}$ be the first index with $B_{\alpha}\cap B_{\alpha_i}\neq\emptyset$ 
(there must be one, or else the process would not have stopped). Hence 
$r_{\alpha}\leq r_{\alpha_{i_0}}$; otherwise, we selected $B_{\alpha_{i_0}}$ incorrectly. But then clearly 
$B(c_{\alpha},r_{\alpha})\subset B(c_{\alpha_{i_0}},3r_{\alpha_{i_0}})$, as desired.
\\
\end{proof}

\begin{lemma}
Suppose $f\in L^{1}(\partial D)$. There exists a positive constant $C=C(N,D)$, such that
\[
	\sigma\left\{y\in\partial D:M[f](y)>t\right\}\leq\frac{C\left\|f\right\|_{L^{1}(\partial D)}}{t},\quad\forall t>0.
\]	
\end{lemma}

\begin{proof}
Choose $t>0$ and let $F$ be a compact subset of $\left\{y\in\partial D:M[f](y)>t\right\}$. Because $\sigma$ is regular, it suffices to estimate $\sigma(F)$. Now for each $y\in F$ there exists $r_y>0$, such that 

\[
	\frac{1}{\sigma(K(y,r_y))}\int_{K(y,r_y)}|f(z)|d\sigma(z)>t.
\]
\\
The balls $\left\{B(y,r_y)\right\}_{y\in F}$ cover $F$. Choose, by Lemma 5.1, finite family of disjoint balls 
$\left\{B(y_i,r_{y_i})\right\}$ so that $\left\{B(y_i,3r_{y_i})\right\}$ cover $F$. Then

\[
	\sigma(F)\leq\sum_i\sigma(K(y_i,3r_{y_i}))\leq 3^{N-1}\frac{c_2}{c_1}\sum_i\sigma(K(y_i,r_{y_i})),
\]
\\
where $c_1,c_2$ are the constants from Lemma 3.3. Denoting $C=3^{N-1}c_2/c_1$ we conclude, that

\[
	\sigma(F)\leq \frac{C}{t}\sum_i\int_{K(y_i,r_{y_i})}|f(z)|d\sigma(z)\leq\frac{C\left\|f\right\|_{L^{1}(\partial D)}}{t}.
\]
\\
\end{proof}

\begin{lemma}
Suppose $u=P_D[f]$, where $f\in L^{1}(\partial D)$, and let $\alpha>0$. Then there exists $C_{\alpha}>0$, 
such that for every $y\in\partial D$

\[
	\sup_{x\in\Gamma_{\alpha}(y)}|u(x)|\leq C_{\alpha}M[f](y).
\]

\end{lemma}

\begin{proof}
Choose $y\in\partial D$ and let $x\in\Gamma_{\alpha}(y)$; denote $\eta=|x-y|$. We have

\[
	|u(x)|\leq\int_{\partial D}P_D(x,z)|f(z)|d\sigma(z)=\int_{|z-y|<2\eta}P_D(x,z)|f(z)|d\sigma(z)
\]

\[
	+\sum^{\infty}_{k=2}\int_{2^{k-1}\eta<|z-y|<2^k\eta}P_D(x,z)|f(z)|d\sigma(z).
\]
\\
By Lemma 3.7, 
\[
	P_D(x,z)\leq \frac{C}{|x-z|^{N-1}}
\]
for some constant $C>0$. Because $\delta(x)\leq|x-z|$, we then get that $P_D(x,z)\leq C(\delta(x))^{1-N}$. 
The cone condition, $|x-y|<(1+\alpha)\delta(x)$, shows, that $\delta(x)>\eta/(1+\alpha)$. Thus 

\[
	\int_{|z-y|<2\eta}P_D(x,z)|f(z)|d\sigma(z)\leq C(1+\alpha)^{N-1}\eta^{1-N}\int_{K(y,2\eta)}|f(z)|d\sigma(z)
\]

\[
	\leq\frac{CC'2^N(1+\alpha)^{N-1}}{2}\cdot\frac{1}{\sigma(K(y,2\eta))}\int_{K(y,2\eta)}|f(z)|d\sigma(z),
\]
\\
where $C'$ is the constant from the upper estimation of Lemma 3.3. 

Similarly, if 
\[
	2^{k-1}\eta<|z-y|<2^k\eta,
\]
where $k\geq2$, then 
\[
	|x-z|\geq|z-y|-|y-x|\geq2^{k-1}\eta-\eta\geq2^{k-2}\eta,
\]
and
\[
	P_D(x,z)\leq \frac{C}{|x-z|^{N-1}}\leq C2^{2N}2^{-kN}\eta^{1-N}.
\]
\\
By Lemma 3.3,
\[
	\eta^{1-N}\leq\frac{C'2^{kN}2^{-k}}{\sigma(K(y,2^k\eta))};
\]
therefore

\[
	\int_{2^{k-1}\eta<|z-y|<2^k\eta}P_D(x,z)|f(z)|d\sigma(z)\leq C2^{2N}2^{-kN}\eta^{1-N}\int_{K(y,2^k\eta)}|f(z)|d\sigma(z)
\]

\[
	\leq\frac{CC'2^{2N}}{2^k}\cdot\frac{1}{\sigma(K(y,2^k\eta))}\int_{K(y,2^k\eta)}|f(z)|d\sigma(z).
\]
\\
Now, denoting $C_{\alpha}=\max\left\{CC'2^N(1+\alpha)^{N-1},CC'2^{2N}\right\}$, we conclude, that 

\[
	|u(x)|\leq C_{\alpha}\sum^{\infty}_{k=1}\frac{1}{2^k\sigma(K(y,2^k\eta))}\int_{K(y,2^k\eta)}|f(z)|d\sigma(z)\leq 
	C_{\alpha}M[f](y).
\]
\\
\end{proof}

\begin{proof2}
For $f\in L^1(\partial D)$ and $\alpha>0$, define the function $T_{\alpha}[f]$ on $\partial D$ by
\[
	T_{\alpha}[f](y)=\limsup_{\Gamma_{\alpha}(y)\ni x\to y}|P_D[f](x)-f(y)|.
\]
\\
We first show, that $T_{\alpha}[f]=0$ almost everywhere on $\partial D$.

Note that 
\[
	T_{\alpha}[f](y)\leq \sup_{x\in\Gamma_{\alpha}(y)}|P_D[f](x)|+|f(y)|\leq C_{\alpha}M[f](y)+|f(y)|
\]
\\
by Lemma 5.3, and that $T_{\alpha}[f_1+f_2]\leq T_{\alpha}[f_1]+T_{\alpha}[f_2]$. Note also that $T_{\alpha}[f]\equiv0$ for every $f\in C(\partial D)$.

Now fix $f\in L^1(\partial D)$ and $\alpha>0$. Also fixing $t\in (0,\infty)$, we wish to show that 
$\sigma(\left\{T_{\alpha}[f]>2t\right\})=0$.

Given $\varepsilon>0$, we may choose $g\in C(\partial D)$ such that $\left\|f-g\right\|_{L^1(\partial D)}<\varepsilon$. Then we have
\[
	T_{\alpha}[f]\leq T_{\alpha}[f-g]+T_{\alpha}[g]=T_{\alpha}[f-g]\leq C_{\alpha}M[f-g]+|f-g|.
\]
\\
Therefore
\[
	\left\{T_{\alpha}[f]>2t\right\}\subset\left\{C_{\alpha}M[f-g]>t\right\}\cup\left\{|f-g|>t\right\}.
\]
By Lemma 5.2 we conclude, that

\[
	\sigma(\left\{T_{\alpha}[f]>2t\right\})\leq\frac{CC_{\alpha}\left\|f-g\right\|_{L^1(\partial D)}}{t}+
	\frac{\left\|f-g\right\|_{L^1(\partial D)}}{t}<\varepsilon\frac{CC_{\alpha}+1}{t}.
\]
\\
Since $\varepsilon$ is arbitrary, we have shown that the set $\left\{T_{\alpha}[f]>2t\right\}$ is contained in sets of 
arbitrarily small measure, and thus $\sigma(\left\{T_{\alpha}[f]>2t\right\})=0$. Because this is true for every $t\in(0,\infty)$, we have proved, that $T_{\alpha}[f]=0$ almost everywhere on $\partial D$.

Now for $k\in\NN$ let $E_k=\left\{T_k[f]=0\right\}$. We have shown that $E_k$ is set of full measure on $\partial D$ for each $k$, and thus $\bigcap_kE_k$ is a set of full measure. At each $y\in\bigcap_kE_k$, $P[f]$ has nontangential limit $f(y)$, which is what we set out to prove. 
\\
\end{proof2}

\section{The Local Fatou Theorem}
\markright{The Local Fatou Theorem}
\medskip
Theorem 5.1 has a local version. We require a definition. A function $u$ on $D$ is said to be $nontangentially$ $bounded$ at $y\in\partial D$ if $u$ is bounded in $\Gamma_{\alpha}(y)$ for some $\alpha>0$. 
\medskip
\begin{theorem}
Suppose $u$ is harmonic on $D$ and $E\subset\partial D$ is the set of points at which $u$ is nontangentially bounded. Then $u$ has a  nontangential limit at almost every point of $E$.
\end{theorem}\medskip

This theorem was originally obtained by Privalov, Plessner, Marcinkiewicz and Zygmund, and Spencer in the classical case $N=2$ by the use of complex-variable techniques. Methods which are effective for the upper half-space in 
$\RR^N$, $N\geq2$, have been served by Calderon and Stein. The proof for this case may be found in \cite{ABR}.

Using similar methods as in \cite{ABR}, we will serve a detailed proof of Theorem 6.1 for the present case, when $D$ is a bounded domain in $\RR^N,N\geq 2$, with the boundary of class $C^2$. In order to do this, we shall use a few important technical lemmas. We begin with some stronger assertion about the behavior of the Poisson kernel inside the cone.

\begin{lemma}
Let $\alpha>0$. There exists positive constant $A_{\alpha}$, such that for every $y\in\partial D$ and 
$x\in\Gamma_{\alpha}(y)$ we have
\[
	P_D(x,y)\geq\frac{A_{\alpha}}{\delta(x)^{N-1}}.
\]

\end{lemma}

\begin{proof}
Notice that, if $K$ is compact set in $D$, then the estimate we seek is trivial for $x\in K$ and $y\in\partial D$, since 
$P_D(x,y)$ is positive and continuous on $D\times\partial D$, and $\delta(x)$ is bounded away from zero on $K$.

Let $r$ be the constant from Lemma 3.2, and let $r_0=r/4$. Hence, for $y\in\partial D$ we have 
\[
	\overline{B}(y+4r_0\nu_y,4r_0)\cap\overline{D}=\left\{y\right\},
\]
\[
	\overline{B}(y-4r_0\nu_y,4r_0)\cap\overline{D^c}=\left\{y\right\}.
\]
\\
Suppose $x\in B(y-r_0\nu_y,r_0)$. This means that

\[
	\left|x-y+r_0\nu_y\right|<r_0
\]

\[
	|x-y|^2-2r_0\langle x-y,-\nu_y\rangle+r_0^2<r_0^2
\]

\[
	|x-y|^2<2r_0\langle x-y,-\nu_y\rangle.
\]
\\
Moreover
\[
	\dist(x,\partial B(y+4r_0\nu_y,4r_0))=|x-y-4r_0\nu_y|-4r_0,
\]

\[
	\dist(x,\partial B(y-4r_0\nu_y,4r_0))=4r_0-|x-y+4r_0\nu_y|,
\]
and thus
\[
	|x-y-4r_0\nu_y|-4r_0=\frac{|x-y-4r_0\nu_y|^2-16r_0^2}{|x-y-4r_0\nu_y|+4r_0}=\frac{|x-y|^2+8r_0
	\langle x-y,-\nu_y\rangle}{|x-y-4r_0\nu_y|+4r_0}
\]

\[
	<\frac{10r_0\langle x-y,-\nu_y\rangle}{8r_0}=\frac{5}{4}\langle x-y,-\nu_y\rangle,
\]

\[
	4r_0-|x-y+4r_0\nu_y|=\frac{16r_0^2-|x-y+4r_0\nu_y|^2}{4r_0+|x-y+4r_0\nu_y|}=\frac{-|x-y|^2+8r_0
	\langle x-y,-\nu_y\rangle}{4r_0+|x-y+4r_0\nu_y|}
\]

\[
	>\frac{6r_0\langle x-y,-\nu_y\rangle}{8r_0}=\frac{3}{4}\langle x-y,-\nu_y\rangle.
\]
Therefore
\[
	|x-y-4r_0\nu_y|-4r_0\leq 2(4r_0-|x-y+4r_0\nu_y|).
\]
Obviously
\[
	\delta(x)\leq \dist(x,\partial B(y+4r_0\nu_y,r_0))=|x-y-4r_0\nu_y|-4r_0,
\]
since $B(y+4r_0\nu_y,4r_0)\subset D^c$, and hence 

\[
	\delta(x)\leq 2(4r_0-|x-y+4r_0\nu_y|).
\]
\\
Now choose $y\in\partial D$ and let $x\in B(y-r_0\nu_y,r_0)$. Denote $c_y=y-4r_0\nu_y$ and $B_y=B(c_y,4r_0)$. Let $G_D$, $G_{B_y}$ 
be the Green's functions for $D$ and $B_y$, respectively. Observe, that $G_{B_y}(x,\cdot)$ is 0 on $\partial B_y$, whereas 
$G_D(x,\cdot)\geq0$ on $\partial B_y$. Since $G_D(x,\cdot)-G_{B_y}(x,\cdot)$ is harmonic on $B_y$, it follows that 
\[
	G_{B_y}(x,t)\leq G_D(x,t),\quad t\in\overline{B_y}.
\]
\[
	G_{B_y}(x,y)=G_D(x,y)=0.
\]
Therefore
\[
	P_{B_y}(x,y)=-\frac{\partial{G_{B_y}(x,y)}}{\partial{\nu_y}}\leq -\frac{\partial{G_D(x,y)}}{\partial{\nu_y}}=P_D(x,y).
\]
\\
The Poisson kernel for $B_y$ has a form 
\[
	P_{B_y}(x,t)=\frac{1}{\omega_{N-1}4r_0}\cdot\frac{16r_0^2-|x-c_y|^2}{|x-t|^N},\quad x\in B_y,t\in\partial B_y.
\]
\\
Hence we have
\[
	P_D(x,y)\geq\frac{1}{\omega_{N-1}4r_0}\cdot\frac{16r_0^2-|x-c_y|^2}{|x-y|^N}
\]

\[
	=\frac{1}{\omega_{N-1}4r_0}\cdot\frac{(4r_0-|x-c_y|)(4r_0+|x-c_y|)}{|x-y|^N}
\]

\[
	\geq \frac{1}{\omega_{N-1}}\cdot\frac{(4r_0-|x-c_y|)}{|x-y|^N}=\frac{1}{\omega_{N-1}}\cdot\frac{(4r_0-|x-y+4r_0\nu_y|)}{|x-y|^N}
	\geq \frac{1}{2\omega_{N-1}}\cdot\frac{\delta(x)}{|x-y|^N}.
\]
\\
Thus we have shown, that  
\[
	P_D(x,y)\geq C\frac{\delta(x)}{|x-y|^N}.
\]
\\
for every $y\in\partial D$ and $x\in B(y-r\nu_y,r_0)$.

Now for $h>0$ let $\Gamma^{h}_{\alpha}(y)=\Gamma_{\alpha}(y)\cap\left\{x:\delta(x)<h\right\}$. We will show, that there exists 
$h_{\alpha}>0$, such that for every $y\in\partial D$, $\Gamma^{h_{\alpha}}_{\alpha}(y)\subset B(y-r_0\nu_y,r_0)$. 
For $y\in\partial D$ denote $\widetilde{B}_y=B(y+4r_0\nu_y,4r_0)$, and recall that 
$\overline{\widetilde{B}}_y\cap\overline{D}=\left\{y\right\}$. Let 
$\widetilde{\Gamma}_{\alpha}(y)$ be the cone in $\widetilde{B}^{c}_{y}$. That is, 
\[
	\widetilde{\Gamma}_{\alpha}(y)=\left\{x\in\widetilde{B}^{c}_{y}:|x-y|<(1+\alpha)\dist(x,\partial\widetilde{B}_y)\right\};
\]
let 
\[
	\widetilde{\Gamma}^{h}_{\alpha}(y)=\widetilde{\Gamma}_{\alpha}(y)\cap\left\{x:\dist(x,\partial\widetilde{B}_y)<h\right\}.
\]
\\
Since $\delta(x)\leq\dist(x,\partial\widetilde{B}_y)$, we have $\Gamma_{\alpha}(y)\subset\widetilde{\Gamma}_{\alpha}(y)$. 
Moreover, observe that

\[
	\Gamma^{h}_{\alpha}(y)=\left\{x\in D:|x-y|<(1+\alpha)\delta(x)\wedge\delta(x)<h\right\}
\]

\[
	=\left\{x\in D:|x-y|<(1+\alpha)\delta(x)\wedge\delta(x)<h\wedge|x-y|<(1+\alpha)h\right\}
\]

\[
	\subset\Gamma_{\alpha}(y)\cap\left\{x\in D:|x-y|<(1+\alpha)h\right\}
\]

\[
	\subset\widetilde{\Gamma}_{\alpha}(y)\cap\left\{x:\dist(x,\partial\widetilde{B}_y)<(1+\alpha)h\right\}
	=\widetilde{\Gamma}^{h'}_{\alpha}(y),
\]
\\
where $h'=(1+\alpha)h$. Thus it suffices to show, that there exists $h'_{\alpha}>0$, such that 
$\widetilde{\Gamma}^{h'_{\alpha}}_{\alpha}(y)\subset B(y-r_0\nu_y,r_0)$ for every $y\in\partial D$. 

Suppose $h'_{\alpha}<8r_0/5(1+\alpha)^2$, and fix $y\in\partial D$, $x\in\widetilde{\Gamma}^{h'_{\alpha}}_{\alpha}(y)$. We have

\[
	|x-y|<(1+\alpha)\dist(x,\partial\widetilde{B}_y)=(1+\alpha)(|x-y-4r_0\nu_y|-4r_0)
\]

\[
	=(1+\alpha)\frac{|x-y-4r_0\nu_y|^2-16r_0^2}{|x-y-4r_0\nu_y|+4r_0}=(1+\alpha)\frac{|x-y|^2-8r_0
	\langle x-y,\nu_y\rangle}{|x-y-4r_0\nu_y|+4r_0}
\]

\[
	\leq (1+\alpha)\frac{|x-y|^2-8r_0\langle x-y,\nu_y\rangle}{8r_0}.
\]
\\
Hence
\[
	\langle x-y,\nu_y\rangle<\frac{|x-y|^2}{8r_0}-\frac{|x-y|}{1+\alpha}
	=\frac{5|x-y|^2}{8r_0}-\frac{|x-y|}{1+\alpha}-\frac{|x-y|^2}{2r_0}
\]

\[
	=|x-y|\left(\frac{5|x-y|}{8r_0}-\frac{1}{1+\alpha}\right)-\frac{|x-y|^2}{2r_0}<-\frac{|x-y|^2}{2r_0},
\]
since $|x-y|<(1+\alpha)h'_{\alpha}$. Thus we conclude
\[
	\langle x-y,\nu_y\rangle<-\frac{|x-y|^2}{2r_0}
\]

\[
	|x-y|^2+2r_0\langle x-y,\nu_y\rangle+r_0^2<r_0^2
\]

\[
	|x-y+r_0\nu_y|<r_0,
\]
\\
what means, that $x\in B(y-r_0\nu_y,r_0)$, so $\widetilde{\Gamma}^{h'_{\alpha}}_{\alpha}(y)\subset B(y-r_0\nu_y,r_0)$. 
Denoting $h_{\alpha}=h'_{\alpha}/(1+\alpha)$, we have $\Gamma^{h_{\alpha}}_{\alpha}(y)\subset B(y-r_0\nu_y,r_0)$.

Therefore, if $x\in\Gamma^{h_{\alpha}}_{\alpha}(y)$, then
\[
	P_D(x,y)\geq C\frac{\delta(x)}{|x-y|^N}>C\frac{\delta(x)}{(1+\alpha)^N\delta(x)^N}
	=\frac{C}{(1+\alpha)^N}\cdot\frac{1}{\delta(x)^{N-1}}.
\]
\\
Now, since $K_{\alpha}=\left\{x\in D:\delta(x)\geq h_{\alpha}\right\}$ is compact, we can find a positive constant 
$A_{\alpha}\leq C/(1+\alpha)^N$, such that for every $y\in\partial D$ and $x\in K_{\alpha}$ 
\[
	P_D(x,y)\geq\frac{A_{\alpha}}{\delta(x)^{N-1}}.
\]
Thus the inequality is proved.

\end{proof}

\begin{lemma}
Suppose $E\subset\partial D$ is Borel measurable, $\alpha>0$, and 
\[
	\Omega = \bigcup_{y\in E}\Gamma_{\alpha}(y).
\]
Then there exists a positive harmonic function $v$ on $D$ such that $v\geq1$ on $(\partial\Omega)\cap D$, and such
that $v$ has nontangential limit $0$ almost everywhere on $E$.
\end{lemma}

\begin{proof}
Define a positive harmonic function $w$ on $D$ by 
\[
	w(x)=\int_{\partial D}{P_D(x,y)\chi_{E^{c}}(y)d\sigma(y)},
\]
where $\chi_{E^{c}}$ denotes the characteristic function of $E^{c}$, the complement of $E$ in $\partial D$.
By Theorem 2, $w$ has nontangential limit $0$ almost everywhere on $E$. We wish to show that $w$ is bounded
away from $0$ on $(\partial\Omega)\cap D$. Choose $x\in(\partial\Omega)\cap D$; this means, that
\[
	|x-y|\geq(1+\alpha)\delta(x)\quad\forall y\in E.
\]
Because $\partial D$ is compact, there exists $x'\in\partial D$, such that $|x-x'|=\delta(x)$. Hence for every 
$y\in E$ we have

\[
	|x-y|\geq(1+\alpha)|x-x'|
\]

\[
	|y-x'|\geq|y-x|-|x-x'|\geq\alpha|x-x'|
\]

\[
	|y-x'|\geq\alpha\delta(x),
\]
\\
so $K(x',\alpha\delta(x))\subset E^{c}$. Therefore
 
\[
	\int_{\partial D}P_D(x,y)\chi_{E^{c}}(y)d\sigma(y)
	=\int_{E^{c}}P_D(x,y)d\sigma(y)\geq\int_{K(x',\alpha\delta(x))}P_D(x,y)d\sigma(y).
\]
\\
On the other hand, if $y\in K(x',\alpha\delta(x))$, then

\[
	|x-y|\leq\delta(x)+|x'-y|<(1+\alpha)\delta(x),
\]
\\
so $x\in\Gamma_{\alpha}(y)$. By Lemmas 3.3 and 6.1, there exist positive constants $c$, $A_{\alpha}$, such that 

\[
	\sigma\left\{K(x',\alpha\delta(x))\right\}\geq c(\alpha\delta(x))^{N-1}
\]
and
\[
	P_D(x,y)\geq \frac{A_{\alpha}}{\delta(x)^{N-1}}\quad\forall y\in K(x',\alpha\delta(x)).
\]
Hence

\[
	\int_{K(x',\alpha\delta(x))}P_D(x,y)d\sigma(y)\geq 
	\frac{A_{\alpha}}{\delta(x)^{N-1}}\sigma\left\{K(x',\alpha\delta(x))\right\}\geq cA_{\alpha}\alpha^{N-1}.
\]
\\
Denoting $c_{\alpha}=cA_{\alpha}\alpha^{N-1}$ (a constant greater than 0 that depends only on $\alpha$ and $N$), 
we see that if $v=w/c_{\alpha}$, then $v$ satisfies the conclusion of the lemma.
\\
\end{proof}

\begin{lemma}
Choose $\alpha,r>0$. Suppose $\nu_1,\nu_2\in S$ and $|\nu_1-\nu_2|<\frac{\alpha}{1+\alpha}$. Then for each $t\in(0,r]$ we have 
the following inequality
\[
	t<(1+\alpha)(r-|t\nu_1-r\nu_2|).
\]
\end{lemma}

\begin{proof}
Let $t\in(0,r]$; because $|\nu_1|=|\nu_2|=1$, we have the following sequence of equivalent inequalities

\[
	t<(1+\alpha)(r-|t\nu_1-r\nu_2|)
\]
\[
	|t\nu_1-r\nu_2|<r-\frac{t}{1+\alpha}
\]
\[
	t^{2}-2rt\langle\nu_1,\nu_2\rangle+r^{2}<\left(r-\frac{t}{1+\alpha}\right)^{2} 
\]
\[
	t^{2}-rt(2-|\nu_1-\nu_2|^{2})<\left(\frac{t}{1+\alpha}\right)^{2}-\frac{2rt}{1+\alpha}
\]
\[
	|\nu_1-\nu_2|^{2}<2-\frac{2}{1+\alpha}-\frac{t}{r}\left(1-\frac{1}{(1+\alpha)^{2}}\right).
\]
Because $t\in(0,r]$, we have
\[
	\left(\frac{\alpha}{1+\alpha}\right)^{2}\leq 2-\frac{2}{1+\alpha}-\frac{t}{r}\left(1-\frac{1}{(1+\alpha)^{2}}\right),
\]	
and this gives the conclusion of the lemma.
\\
\end{proof}

Now let $r$ be the constant from Lemma 3.2 and let $r_0=r/4$. We may assume, that $r_0\leq1$. For $y\in\partial D$ denote
\[
	B_y=B(y-r_0\nu_y,r_0),\quad B'_y=B(y+r_0\nu_y,r_0)
\] 
and
\[
	\Gamma'_{\alpha}(y)=\left\{x\in B_y:|x-y|<(1+\alpha)\dist(x,\partial B_y)\right\}.
\]
Obviously 
\[
	\dist(x,\partial B_y)=r_0-|y-r_0\nu_y-x|,
\]
and because $B_y\subset D$, we have 
\[
	\Gamma'_{\alpha}(y)\subset\Gamma_{\alpha}(y).
\]

\begin{lemma}
Let $\alpha>0$.There exists a positive constant $d=d(\alpha)$ such that for every $x,y\in\partial D$ and $|x-y|\leq d$ we have
\medskip
\begin{enumerate}
	\item $\forall t\in (0,r_0]$
	\[
		y-t\nu_x\in\Gamma'_{\alpha}(y).
	\]	
	\item $\forall t,s\in (0,r_0]$, $t<s$
	\[
		\delta(y-s\nu_x)-\delta(y-t\nu_x)\geq\frac{s-t}{1+\alpha}.
	\]	
\end{enumerate}
\end{lemma}

\begin{proof}

Recall that 
\[
	|\nu_x-\nu_y|\leq c_0|x-y| \quad \forall x,y\in\partial D. 
\]
\\
We may assume, that $c_0\geq 1$. So let $x,y\in\partial D$ and suppose
\[
	|x-y|<\frac{\alpha}{c_0(1+\alpha)};
\]  
then
\[
	|\nu_x-\nu_y|<\frac{\alpha}{1+\alpha},
\]
\\
and because $\nu_x,\nu_y\in S$, by Lemma 6.3 we have

\[
	|y-t\nu_x-y|=t<(1+\alpha)(r_0-|t\nu_x-r_0\nu_y|)=(1+\alpha)(r_0-|y-r_0\nu_y-(y-t\nu_x)|)
\]

\[
	=(1+\alpha)\dist(y-t\nu_x,\partial B_y) \quad \forall t\in (0,r_0].
\]
\\
This gives 1. 

Now let $\pi$ be the orthogonal projection; then for every $x,y\in D_{r_0}$ we have

\[
	|\pi(x)-\pi(y)|\leq 4|x-y|,
\]
\\
by Lemma 3.4. So let $x,y\in\partial D$ and assume 

\[
	|x-y|< \frac{\alpha}{8c_0^2(1+\alpha)}.
\]
Then we have
\[
	\langle\nu_x,\nu_y\rangle=1-\frac{|\nu_x-\nu_y|^{2}}{2}\geq1-\frac{c_0^{2}|x-y|^{2}}{2}>0,
\]
\\
and choosing $t\in (0,r_0)$ we conclude

\[
	|y-t\nu_x-(y-t\langle\nu_x,\nu_y\rangle\nu_y)|=t|\langle\nu_x,\nu_y\rangle\nu_y-\nu_x|=
	t\left(1-\langle\nu_x,\nu_y\rangle^{2}\right)^{\frac{1}{2}}
\]

\[
	=t\left[1-\left(1-\frac{|\nu_x-\nu_y|^{2}}{2}\right)^{2}\right]^{\frac{1}{2}}
	=t\left(|\nu_x-\nu_y|^{2}-\frac{|\nu_x-\nu_y|^{4}}{4}\right)^{\frac{1}{2}}
\]

\[
	=t|\nu_x-\nu_y|\left(1-\frac{|\nu_x-\nu_y|^{2}}{4}\right)^{\frac{1}{2}}\leq r_0|\nu_x-\nu_y|\leq c_0|x-y|
	< \frac{\alpha}{8c_0(1+\alpha)}.
\]
\\
In particular, since $t\langle\nu_x,\nu_y\rangle\in (0,r_0)$, we have
	
\[
	|\pi(y-t\nu_x)-\pi(y-t\langle\nu_x,\nu_y\rangle\nu_y)|=|\pi(y-t\nu_x)-y|<\frac{\alpha}{2c_0(1+\alpha)}.
\]
\\
Therefore

\[
	|\nu_{\pi(y-t\nu_x)}-\nu_y|\leq c_0|\pi(y-t\nu_x)-y|<\frac{\alpha}{2(1+\alpha)}.
\]
\\

Let $s\in(t,r_0]$; because

\[
	y-s\nu_x=\pi(y-t\nu_x)-\delta(y-t\nu_x)\nu_{\pi(y-t\nu_x)}-(s-t)\nu_x,
\]
\\
and

\[
	|\nu_{\pi(y-t\nu_x)}-\nu_x|\leq |\nu_{\pi(y-t\nu_x)}-\nu_y|+|\nu_y-\nu_x|<\frac{\alpha}{1+\alpha},
\]
\\
where $\nu_{\pi(y-t\nu_x)},\nu_x\in S$, by Lemma 6.3 we have

\[
	y-s\nu_x\in B(\pi(y-t\nu_x)-[\delta(y-t\nu_x)+r_0]\nu_{\pi(y-t\nu_x)},r_0).
\]
\\
Because
\[
	B(\pi(y_t)-(\delta_t+r_0)\nu_{\pi(y_t)},r_0)\subset
	B(\pi(y_t)-(\delta_t+r_0)\nu_{\pi(y_t)},\delta_t+r_0)
\]

\[
	\subset B(\pi(y_t)-2r_0\nu_{\pi(y_t)},2r_0)\subset D,\quad \delta_t:=\delta(y-t\nu_x),\quad y_t:=y-t\nu_x,
\]
\\
we have

\[
	\delta(y-s\nu_x)-\delta(y-t\nu_x)\geq r_0-|y-s\nu_x-[\pi(y-t\nu_x)-(\delta(y-t\nu_x)+r_0)\nu_{\pi(y-t\nu_x)}]|
\]

\[
	=r_0-|r_0\nu_{\pi(y-t\nu_x)}-(s-t)\nu_x|;
\]
\\
one more time by Lemma 6.3 we conclude

\[
	r_0-|r_0\nu_{\pi(y-t\nu_x)}-(s-t)\nu_x|>\frac{s-t}{1+\alpha},
\]
\\
what gives 2.

\end{proof}

For $\alpha>0$ let $d_{\alpha}$ be the constant from Lemma 6.4; assume additionally, that  $d_{\alpha}<\rho$, where $\rho$ is 
the constant from Lemma 3.1. Because $\partial D$ is compact, there exists a finite family of balls 
$\left\{B(y_i,d_{\alpha}):i=1,...,m\right\}$, where $y_i\in\partial D$ for each $i=1,...,m$, such that 
\[
	\partial D\subset\bigcup^{m}_{i=1}B\left(y_i,d_{\alpha}\right);
\]
hence
\[
	\partial D=\bigcup^{m}_{i=1}K\left(y_i,d_{\alpha}\right).
\]
\medskip

For $i\in\left\{1,...,m\right\}$ denote $\nu_i:=\nu_{y_i}$ and for $r_0$ as before, let

\[
	D^{\alpha}_{i}=\left\{y-t\nu_i:y\in K\left(y_i,d_{\alpha}\right),t\in\left(0,\frac{r_0}{2}\right)\right\}.
\]

\begin{lemma}
For every $\alpha>0$ and $i\in\left\{1,...,m\right\}$ we have
\begin{enumerate}
	\item $D^{\alpha}_{i}$ is an open subset of $D$ and $(\partial D^{\alpha}_{i})\cap(\partial D)=\overline{K}(y_i,d_{\alpha})$.
	\item Let $E$ be the closed subset of $\partial D$ and suppose $E\subset K\left(y_i,d_{\alpha}\right)$.
	Let 
	\[
		\Omega = \bigcup_{y\in E}\Gamma_{\alpha}(y);
	\]
	then there exists $\varepsilon_0>0$, such that 
	\[
		\Omega\cap\left\{x\in D:\delta(x)<\varepsilon_0\right\}\subset D^{\alpha}_{i}.
	\]
	
\end{enumerate}

\end{lemma}

\begin{proof}
Choose $\alpha>0$ and $i\in\left\{1,...,m\right\}$. Without loss of generality we may assume, that $y_i=0$ and 
$\nu_i=e_N=(0,...,0,1)$. Obviously 

\[
	\overline{D^{\alpha}_{i}}=
	\left\{y-te_N:y\in\overline{K}\left(0,d_{\alpha}\right),t\in\left[0,\frac{r_0}{2}\right]\right\}.
\]
\\
By Lemma 6.4,
 
\[
	\forall y\in \overline{K}\left(0,d_{\alpha}\right)\forall t\in(0,r_0]\quad y-te_N\in \Gamma'_{\alpha}(y),
\]
\\
in particular $y-te_N\in B_y$. Hence $D^{\alpha}_{i}\subset D$. Moreover, 
$\overline{D^{\alpha}_{i}}\backslash\overline{K}(0,d_{\alpha})\subset D$, 
and hence $(\partial D^{\alpha}_{i})\cap(\partial D)=\overline{K}(0,d_{\alpha})$.

Now let $x\in D^{\alpha}_{i}$. Then $x=y-te_N$, where $y\in K\left(0,d_{\alpha}\right)$ and 
${t\in\left(0,\frac{r_0}{2}\right)}$. By Lemma 3.1, since $d_{\alpha}<\rho$, $y=(\overline{y},g(\overline{y}))$, 
where $\overline{y}\in\RR^{N-1}$ is the projection of $y$ into $\RR^{N-1}\times\left\{0\right\}$ (since $\nu_0=e_N$), 
and $g$ is a real-valued, $C^2$ function defined on $B_{N-1}(0,\rho)$. Strictly, 
$\overline{K}\left(0,d_{\alpha}\right)$ is a part of the graph of function $g$. Moreover, since $\nu_0=e_N$, we may assume that 

\[
	B_N(0,\rho)\cap D=B_N(0,\rho)\cap\left\{(\overline{x},x_N):
	\overline{x}\in B_{N-1}(\overline{0},\rho)\wedge x_N<g(\overline{x})\right\}.
\]
\\
Therefore 
$x=(\overline{y},g(\overline{y})-t)$, and $\overline{x}=\overline{y}$. Then there exists $\varepsilon_1>0$, such that if 
$|x-x'|<\varepsilon_1$, then $y'=(\overline{x}',g(\overline{x}'))\in K\left(0,d_{\alpha}\right)$. Moreover, 
$x'=y'-(g(\overline{x}')-x_N')e_N$, and there exists $\varepsilon_2>0$, such that 
$t'=g(\overline{x}')-x_N'\in\left(0,\frac{r_0}{2}\right)$ whenever $|x-x'|<\varepsilon_2$. Thus if 
$\varepsilon_3=\min\left\{\varepsilon_1,\varepsilon_2\right\}$, and $|x-x'|<\varepsilon_3$, then $x'=y'-t'e_N$, where 
$y'\in K\left(0,d_{\alpha}\right)$ and $t'\in\left(0,\frac{r_0}{2}\right)$, so $B(x,\varepsilon_3)\subset D^{\alpha}_{i}$. 
This gives 1. 

Now suppose $E\subset K\left(0,d_{\alpha}\right)$, $E=\overline{E}$ and let 

\[
	\Omega = \bigcup_{y\in E}\Gamma_{\alpha}(y).
\]
\\
By Lemma 3.1, there exists $M>0$, such that $|g(\overline{x})-g(\overline{y})|\leq M|\overline{x}-\overline{y}|$ for every 
$\overline{x},\overline{y}\in B_{N-1}(\overline{0},\rho)$. Let $x\in\Omega$ and suppose 

\[
	\delta(x)<\varepsilon_0=\min\left\{\frac{\dist\left[E,\partial B\left(0,d_{\alpha}\right)\right]}{\sqrt{1+M^2}(1+\alpha)},
	\frac{r_0}{2(1+M)(1+\alpha)}\right\}.
\]
\\
Then for some $y\in E$ we have 

\[
	|x-y|<(1+\alpha)\delta(x)<\dist\left[E,\partial B\left(0,d_{\alpha}\right)\right].
\]
Thus
\[
	|x|\leq|x-y|+|y|<\dist\left[E,\partial B\left(0,d_{\alpha}\right)\right]+|y|\leq d_{\alpha}<\rho,
\]
\\
in particular $|\overline{x}|<\rho$. Hence $x=(\overline{x},g(\overline{x}))-(g(\overline{x})-x_N)e_N$, and it suffices to show 
that $(\overline{x},g(\overline{x}))\in K\left(0,d_{\alpha}\right)$ and $g(\overline{x})-x_N\in\left(0,\frac{r_0}{2}\right)$. 
We have

\[
	|(\overline{x},g(\overline{x}))|\leq |(\overline{x},g(\overline{x}))-(\overline{y},g(\overline{y}))|+|y|=
	\sqrt{|\overline{x}-\overline{y}|^2+|g(\overline{x}))-g(\overline{y})|^2}+|y|
\]

\[
	\leq\sqrt{1+M^2}|\overline{x}-\overline{y}|+|y|\leq\sqrt{1+M^2}|x-y|+|y|
\]

\[
	<\sqrt{1+M^2}(1+\alpha)\delta(x)+|y|<\dist\left[E,\partial B\left(0,d_{\alpha}\right)\right]+|y|\leq d_{\alpha}.
\]
\\
Moreover, $g(\overline{x})-x_N>0$ since $x\in B(0,\rho)\cap D$, and

\[
	g(\overline{x})-x_N\leq |g(\overline{x})-g(\overline{y})|+|y_N-x_N|\leq
	M|\overline{x}-\overline{y}|+|x-y|\leq(M+1)|x-y|<\frac{r_0}{2}.
\]
\\
Therefore $x\in D^{\alpha}_{i}$, what gives 2.
\\
\end{proof}

The next lemma is the crux of the proof of Theorem 6.1.

\begin{lemma}
Let $E\subset\partial D$ be Borel measurable, let ${\alpha>0}$, and let \[\Omega = \bigcup_{y\in E}\Gamma_{\alpha}(y).\]
Suppose $u$ is harmonic on $D$ and bounded on $\Omega$. Then for almost every $y\in E$, $\lim u(x)$ exists as 
$x\rightarrow y$ within $\Gamma_{\alpha}(y)$.
\end{lemma}

\begin{proof}
Because
\[
	\partial D=\bigcup^{m}_{i=1}K\left(y_i,d_{\alpha}\right),
\]
it suffices to show, that the conclusion of the lemma holds for the set
\[
	E_i=K\left(y_i,d_{\alpha}\right)\cap E,
\]
where $i\in\left\{1,...,m\right\}$ is fixed. We may also assume that $u$ is real valued and $|u|\leq1$ on $\Omega$.
Suppose first, that $E_i$ is closed, and let 
\[
	\Omega_i = \bigcup_{y\in E_i}\Gamma_{\alpha}(y).
\]
For $k\in\NN$, let

\[
	E^{k}_{i}=\left\{\left[K\left(y_i,d_{\alpha}\right)-\frac{r_0}{k}\nu_i\right]\cap\Omega_i\right\}
	+\frac{r_0}{k}\nu_i.
\]
\\
Because $\Omega_i$ is an open set in $\RR^{N}$, and $K(y_i,d_{\alpha})$ is an open subset of $\partial D$, 
$E^{k}_{i}$ is an open subset of $\partial D$ too. Obviously $E^{k}_{i}\subseteq K(y_i,d_{\alpha})$. By Lemma 6.4, 
for each $y\in E_i$ $y-\frac{r_0}{k}\nu_i\in\Gamma_{\alpha}(y)$; hence $E_i\subseteq E^{k}_{i}$ for every $k$. Let 

\[
	f^{k}_{i}(y)=\chi_{E^{k}_{i}}(y)u\left(y-\frac{r_0}{k}\nu_i\right),\quad y\in\partial D.
\]
\\
Because $y-\frac{r_0}{k}\nu_i\in\Omega$ if $y\in E^{k}_{i}$, $|f^{k}_{i}|\leq1$ on $\partial D$ for every $k$.
By Banach-Alaoglu theorem, there exists a subsequence, which we still denote by $(f^{k}_{i})$, that converges weak$^{\ast}$ 
to some $f_i\in L^{\infty}(\partial D)$. Because $E^{k}_{i}$ is open, each $f^{k}_{i}$ is continuous on $E^{k}_{i}$, and 
thus, by the properties of the Poisson kernel listed in section 3, $P_{D}[f^{k}_{i}]$ extends continuously to $D\cup E^{k}_{i}$.

By Lemma 6.5, $D^{\alpha}_{i}$ is an open subset of $D$, additionally 
$(\partial D^{\alpha}_{i})\cap(\partial D)=\overline{K}(y_i,d_{\alpha})$. Now if $x\in\overline{D^{\alpha}_{i}}\cap D$, 
then

\[
	x-\frac{r_0}{k}\nu_i=y-\left(t+\frac{r_0}{k}\right)\nu_i,
\]
\\
where $y\in\overline{K}(y_id_{\alpha})$ and $t\in\left(0,\frac{r_0}{2}\right]$; $t+\frac{r_0}{k}\leq r_0$ for $k\geq2$, 
thus by Lemma 6.4, $x-\frac{r_0}{k}\nu_i\in D$. Therefore the function $u^{k}_{i}$ given by

\[
	u^{k}_{i}(x)=P_{D}[f^{k}_{i}](x)-u\left(x-\frac{r_0}{k}\nu_i\right),\quad k\geq2,
\]
\\
is harmonic on $D^{\alpha}_{i}$ and continuous on $\overline{D^{\alpha}_{i}}\cap D$. Moreover, $u^{k}_{i}$ extends continuously 
to $(\overline{D^{\alpha}_{i}}\cap D)\cup E^{k}_{i}$, with $u^{k}_{i}=0$ on $E^{k}_{i}$.  We have assumed, that $E_i$ 
is closed, and thus $\overline{\Omega}_i\cap(\partial D)=E_i$; by Lemma 6.5, there exists $\varepsilon_0>0$ such that 

\[
	\Omega_i|_{\varepsilon_0}:= \Omega_i\cap\left\{x\in D:\delta(x)<\varepsilon_0\right\}\subset D^{\alpha}_{i}.
\]
\\
Therefore we conclude, that in particular $u^{k}_{i}$ is harmonic on $\Omega_i|_{\varepsilon_0}$ and 
continuous on $\overline{\Omega_i|_{\varepsilon_0}}$ with $u^{k}_{i}=0$ on $E_i$.

Let $x\in\Omega_i|_{\varepsilon_0}$; then $x=y-t\nu_i$, where $y\in K(y_i,d_{\alpha})$, $t\in\left(0,\frac{r_0}{2}\right)$ and there exists $a\in E_i$, such that $x\in\Gamma_{\alpha}(a)$. For $k\geq2$ we have

\[
	\left|x-\frac{r_0}{k}\nu_i-a\right|=\left|y-t\nu_i-a-\frac{r_0}{k}\nu_i\right|<(1+\alpha)\delta(y-t\nu_i)+\frac{r_0}{k}
\]

\[
	\leq (1+\alpha)\delta\left(y-\left(t+\frac{r_0}{k}\right)\nu_i\right),
\]
\\
where the last inequality is the result of Lemma 6.4. This means, that $x-\frac{r_0}{k}\nu_i\in\Omega_i$ for $k\geq2$, and because $|u|\leq1$ on $\Omega$, we have $|u^{k}_{i}|\leq2$ on $\overline{\Omega_i|_{\varepsilon_0}}$.

Now let $v$ be the function of Lemma 6.2 with respect to $\Omega_i$. Because $v\geq1$ on 
$(\partial\Omega_i)\cap(\partial(\Omega_i|_{\varepsilon_0}))\cap D$, and there exists constant $c>0$, such that $v\geq c$ on 
$\left\{x\in D:\delta(x)=\varepsilon_0\right\}$, thus the function 

\[
	\tilde{v}(x)=\frac{v(x)}{\min\left\{c,1\right\}}
\]
\\
has the same properties as $v$, but with respect to $\Omega_i|_{\varepsilon_0}$. Then 

\[
	\liminf_{x\rightarrow\partial(\Omega_i|_{\varepsilon_0})}(2\tilde{v}-u^{k}_{i})(x)\geq0.
\]
\\
By the minimum principle, $2\tilde{v}-u^{k}_{i}\geq0$ in $\Omega_i|_{\varepsilon_0}$. Letting $k\rightarrow\infty$, we then see 
that $2\tilde{v}-(P_{D}[f_{i}]-u)\geq0$ in $\Omega_i|_{\varepsilon_0}$. Because this argument applies as well to  $2\tilde{v}+u^{k}_{i}$, we conclude that $|P_{D}[f_{i}]-u|\leq 2\tilde{v}$ in $\Omega_i|_{\varepsilon_0}$.

By Theorem 5.1, $P_{D}[f_{i}]$ has nontangential limits almost everywhere on $\partial D$, while Lemma 6.2 asserts $\tilde{v}$ has nontangential limits $0$ almost everywhere on $E_i$. From this and the last inequality, the desired limits for $u$ follow.

Now if $E_i$ is any Borel measurable and bounded set, then for each $k\in\NN$ there exists a closed set $F_k\subseteq E_i$, 
such that $\sigma(E_i\backslash F_k)<\frac{1}{k}$. Let $A_k$ be the set of points $y\in F_k$, such that 
$\lim u(x)$ doesn't exists as $x\rightarrow y$ within $\Gamma_{\alpha}(y)$; then 
	
\[
	\sigma\left(\bigcup^{\infty}_{k=1}A_k\right)\leq\sum^{\infty}_{k=1}\sigma(A_k)=0,
\]
\\
and because $\sigma(E_i\backslash\bigcup^{\infty}_{k=1}F_k)=0$, the conclusion of the lemma holds with respect to $E_i$.
\\
\end{proof}

Let $E\subset\partial D$ be Borel measurable; a point $y\in E$ is said to be a $point$ $of$ $density$ of $E$ provided
\[
	\lim_{r\rightarrow 0}\frac{\sigma(K(y,r)\cap E)}{\sigma(K(y,r))}=1.
\]
\\
By the Lebesque Differentiation Theorem, almost every point of $E$ is a point of density of $E$.

\begin{lemma}
Suppose $E\subset\partial D$ is Borel measurable, $\alpha>0$, and
\[
\Omega = \bigcup_{y\in E}\Gamma_{\alpha}(y).
\]
Suppose $u$ is continuous on $D$ and bounded on $\Omega$. If $y$ is a point of density of $E$, then u is bounded in 
$\Gamma_{\beta}(y)$ for every $\beta>0$.
\end{lemma}

\begin{proof}
Let $y$ be a point of density of $E$, and let $\beta>0$. For $h>0$ denote

\[
	\Gamma^{h}_{\beta}(y)=\left\{x\in D:|x-y|<(1+\beta)\delta(x)\wedge\delta(x)<h\right\}.
\]
\\
It suffices to show that $\Gamma^{h}_{\beta}(y)\subset\Omega$ for some $h>0$.

Let $c_1,c_2$ be the constants from Lemma 3.3; there exists $r_1>0$ such that for each $r<r_1$ we have

\[
	\frac{\sigma(K(y,r)\cap E)}{\sigma(K(y,r))}>1-\frac{c_1}{c_2}\left(\frac{\alpha}{2+\alpha+\beta}\right)^{N-1}.
\]
\\
So let $h_1=\frac{r_1}{2+\alpha+\beta}$ and fix $x\in\Gamma^{h_1}_{\beta}(y)$. There exists $x'\in\partial D$, such that 
\[
	|x-x'|=\delta(x).
\] 
Suppose $y'\in K(x',\alpha\delta(x))$; we have

\[
	|y'-x'|<\alpha|x-x'|
\]

\[
	|y'-x|\leq |y'-x'|+|x'-x|<(1+\alpha)|x-x'|,
\]
\\
what means, that $x\in\Gamma_{\alpha}(y')$. Hence it suffices to show, that $K(x',\alpha\delta(x))\cap E$ is non-empty.

Because $x\in\Gamma^{h_1}_{\beta}(y)$, we have

\[
	|y-x'|-|x'-x|\leq |y-x|<(1+\beta)|x-x'|,
\]
\\
and thus $x'\in K(y,(2+\beta)\delta(x))$, what implies, that 
\[
	K(x',\alpha\delta(x))\subset K(y,(2+\alpha+\beta)\delta(x)).
\]
Moreover, because $\delta(x)<h_1$, we conclude

\[
	\frac{\sigma\left\{K(y,(2+\alpha+\beta)\delta(x))\cap E\right\}}{\sigma\left\{K(y,(2+\alpha+\beta)\delta(x))\right\}}
	>1-\frac{c_1}{c_2}\left(\frac{\alpha}{2+\alpha+\beta}\right)^{N-1} 
\]
\\
\[
	\geq 1-\frac{\sigma\left\{K(x',\alpha\delta(x))\right\}}{\sigma\left\{K(y,(2+\alpha+\beta)\delta(x))\right\}}
	\geq\frac{\sigma\left\{K(y,(2+\alpha+\beta)\delta(x))\backslash
	K(x',\alpha\delta(x))\right\}}{\sigma\left\{K(y,(2+\alpha+\beta)\delta(x))\right\}},
\]
\\
and hence
\[
	\sigma\left\{K(y,(2+\alpha+\beta)\delta(x))\cap E\right\}>
	\sigma\left\{K(y,(2+\alpha+\beta)\delta(x))\backslash K(x',\alpha\delta(x))\right\}.
\]
\\
Therefore $K(x',\alpha\delta(x))\cap E$ is non-empty, and thus $\Gamma^{h_1}_{\beta}(y)\subset\Omega$, as desired.
\\
\end{proof}

\begin{lemma}
Let $u$ be continuous in $D$ and suppose $E\subset\partial D$ is the set of points at which $u$ is nontangentially bounded. Then $E$ is Borel measurable, and for any $\varepsilon>0$, there exists a closed set $E_{0}\subset E$, such that $\sigma(E\backslash E_{0})<\varepsilon$ and $u$ is bounded in \[\Omega = \bigcup_{y\in E_{0}}\Gamma_{\alpha}(y)\] for every $\alpha>0$.
\end{lemma}

\begin{proof}

For $k=1,2,...$, set $E_k=\{y\in\partial D:|u|\leq k$ on $\Gamma_{\frac{1}{k}}(y)\}$. Then each $E_k$ is closed, and 
\[
	E=\bigcup^{\infty}_{k=1}E_k
\] 
what means, that the set $E$ is Borel measurable. Fix $k\in\NN$; applying Lemma 6.7 to $E_k$, 
and recalling that the points of density of $E_k$ form a set of full measure in $E_k$, we see that there is a set $F_k\subset E_k$, $\sigma(E_k\backslash F_k)=0$, such that $u$ is bounded in $\Gamma_{\alpha}(y)$ for every $y\in F_k$ and every $\alpha>0$. Let
\[
	F=\bigcup^{\infty}_{k=1}F_k;
\] 
obviously $\sigma(E\backslash F)=0$. For fixed positive integer $k$, we can write $F$ as 
$F=\bigcup_j F^{k}_{j}$, where $F^{k}_{j}=\{y\in F:|u|\leq j$ on $\Gamma_{k}(y)\}$. 
Now fix $\varepsilon>0$. Because $F^{k}_{j}\subset F^{k}_{j+1}$ for each $j\in\NN$, there exists $j_k$, such that 
$\sigma(E\backslash F^{k}_{j_k})=\sigma(F\backslash F^{k}_{j_k})<\varepsilon/2^k$. Hence for each $k$, $u$ is bounded on the set
\[
	\Omega_k = \bigcup_{y\in F^{k}_{j_k}}\Gamma_{k}(y).
\]
Let
	\[E_0=\bigcap^{\infty}_{k=1}F^{k}_{j_k}.
\]
Then 

\[
	\sigma\left(E\backslash E_0\right)=\sigma\left(E\cap\left(\bigcap^{\infty}_{k=1}F^{k}_{j_k}\right)^C \right)
	=\sigma\left(\bigcup^{\infty}_{k=1}(E\backslash F^{k}_{j_k})\right)\leq\sum^{\infty}_{k=1}\sigma(E\backslash F^{k}_{j_k})
\]

\[
	<\sum^{\infty}_{k=1}\varepsilon/2^k=\varepsilon,
\]
\\
and it is easily seen, that $u$ is bounded in 

\[
	\Omega = \bigcup_{y\in E_0}\Gamma_{\alpha}(y)
\]
for every $\alpha>0$.

\end{proof}

Now we are ready to prove Theorem 6.1.
\\
\begin{proof3}
By Lemma 6.8, $E$ is Borel measurable, and for each $k\in\NN$ we can choose $E_k\subset E$, such that 
$\sigma(E\backslash E_k)<\frac{1}{k}$ and u is bounded on the set
\[
	\Omega_k = \bigcup_{y\in E_k}\Gamma_{j}(y)
\]
for every $j\in\NN$. Fix $j,k$; by Lemma 6.6, for almost every $y\in E_k$ $\lim u(x)$ exists as $x\rightarrow y$ within $\Gamma_{j}(y)$. Letting $k\rightarrow\infty$, we then get the same result for almost every $y\in E$. Now $u$ has nontangential 
limit $L$ at $y$ if and only if $\lim u(x)=L$ as $x\rightarrow y$ within $\Gamma_{j}(y)$ for every $j\in\NN$, and thus $u$ has 
nontangential limit almost everywhere on $E$, as desired.
\\
\end{proof3}

\section{The Area Theorem}
\markright{The Area Theorem}
\medskip
The question whether $u$ has nontangential limits can also be answered in terms of the "area integral". As before we assume, that $D$ is a bounded domain with the boundary of class $C^2$. For any harmonic function $u$ on $D$, 
$\alpha>0$ and a point $y\in\partial D$, we consider $S_{\alpha}u(y)$ by

\[
	S_{\alpha}u(y)=\int_{\Gamma_{\alpha}(y)}|\nabla u(x)|^{2}(\delta(x))^{2-N}dx.
\]
\\
A necessary and sufficient condition for the existence of nontangential limits, can be formulated as follows.

\begin{theorem}
Let u be harmonic on $D$.
\begin{enumerate}
\item 
	Suppose $E\subset\partial D$ is the set of points at which $u$ is nontangentially bounded. Then for a.e. $y\in E$,
	$S_{\alpha}u(y)$ is finite for every $\alpha>0$.
\item
	Conversely, suppose $E$ is the set of points $y\in\partial D$, such that $S_{\alpha}u(y)$ is finite, where $\alpha$ can 
	depend on $y$. Then $u$ has a nontangential limit at almost every point of $E$.
\end{enumerate}
\end{theorem}

This theorem has been proved by Stein in \cite{S2}, in the case, when $D$ is a half-space in $\RR^N$. By the use of very similar methods, we will prove the necessity part (1) of it in the present case. Because of some technical 
difficulties, we omit the sufficiency (2). The detailed proof of this part (by the use of other techniques) may be 
found in \cite{W}, together with various  generalizations. 

We begin with some technical lemmas. Let $\alpha$ and $\beta$ be given positive quantities with $\alpha<\beta$.
\newpage
\begin{lemma}
Let $u$ be harmonic in the cone $\Gamma_{\beta}(y)$ and suppose that $|u|\leq 1$ there. Then

\[
	\delta(x)|\nabla u(x)|\leq A \quad \forall x\in\Gamma_{\alpha}(y),
\]
\\
where $A=A(\alpha,\beta)$ depends only on the indicated parameters but not on $y$ or $u$.
\end{lemma} 

\begin{proof}
We shall need the following fact: if $u$ is harmonic in $B(0,r)$ and $|u|\leq 1$ there, then $|\nabla u(0)|\leq A/r$, where $A$ does not depend on $u$ (see \cite{ABR}). 

Let $x$ be any point in $\Gamma_{\alpha}(y)$. Notice that since $\alpha<\beta$, there exists a fixed constant 
$c>0$, which does not depend on $x$ and $y$, so that 
\[
	B(x,c\delta(x))\subset\Gamma_{\beta}(y); 
\]
it suffices to take 
\[
	c<\frac{\beta-\alpha}{1+\beta}.
\]

We now apply the previous fact to $u$ and the ball $B(x,c\delta(x))$, and obtain
\[
	|\nabla u(x)|\leq\frac{A}{c\delta(x)}.
\]
Hence
\[
	\delta(x)|\nabla u(x)|\leq A/c \quad \forall x\in \Gamma_{\alpha}(y).
\]
\\
\end{proof}

For positive constants $\alpha,h$ and $y\in\partial D$ recall the notation
 
\[
	\Gamma^{h}_{\alpha}(y)=\left\{x\in D:|x-y|<(1+\alpha)\delta(x),\delta(x)<h \right\}.
\]
\\
Let $E\subset\partial D$ be Borel measurable. For $\rho>0$ and $a\in\partial D$ let $E(a,\rho)=E\cap \overline{B}(a,\rho)$.
For positive $\alpha,h$ denote
\[
	\Omega(a,h,\rho,\alpha) = \bigcup_{y\in E(a,\rho)}\Gamma^{h}_{\alpha}(y).
\]

\newpage
\begin{lemma}
Suppose $E$ is closed subset of $\partial D$ and let $\alpha>0$. There exist positive constants $h,\rho$, such that 
for each $a\in\partial D$ we can choose a family of regions $\left\{\Omega_n\right\}^{\infty}_{n=1}$ with the following properties:
\begin{enumerate}
	\item $\overline{\Omega}_{n}\subset\Omega(a,h,\rho,\alpha)\quad$ $\forall n$ 
	\item $\Omega_{n}\subset\Omega_{n+1}\quad$ $\forall n$
	\item $\Omega(a,h,\rho,\alpha)=\bigcup\Omega_n$
	\item The boundary $B_n$ of $\ \Omega_n$ is piecewise of class $C^2$, at a positive distance from 
				$\partial D$, and there exists a constant $A>0$ such that $\sigma_n(B_n)<A$ $\forall n$.
\end{enumerate}
\end{lemma}

\begin{proof}
Recall, that for every $x,y\in\partial D$
\[
	|\nu_x-\nu_y|\leq c_0|x-y|.
\]
Let $E$ be closed subset of $\partial D$, $\alpha>0$. Choose $a\in\partial D$; without loss of generality 
we may assume that $-\nu_a=e_N=(0,...,0,1)$. Let $d_{\alpha}$ be the constant from Lemma 6.4, and let 
$\rho,h$ be positive constants, such that

\[
	(2+\alpha)h+\rho\leq\min\left\{d_{\alpha},\frac{1}{c_0}\sqrt{\frac{\alpha}{1+\alpha}}\right\}.
\]
\\
Assume additionally, that $h$ satisfies the condition of Corollary 3.1. 

Now observe, that 
\[
	\Omega(a,h,\rho,\alpha) = \bigcup_{y\in E(a,\rho)}\Gamma^{h}_{\alpha}(y)
	=\left\{x\in D:\widetilde{\delta}(x)<(1+\alpha)\delta(x),\quad \delta(x)<h\right\},
\] 
where $\widetilde{\delta}(x)=\dist(x,E(a,\rho))$. Let $\varphi_t$ be a $C^{\infty}$ "approximation to the identity". It may be constructed as follows. Take a function $\varphi$ of class $C^{\infty}$ on $\RR^N$, such that 
$\varphi\geq0$, $\varphi\equiv0$ on $\RR^N\backslash B(0,1)$ and 
\[
	\int_{\RR^{N}}\varphi(x)dx=1.
\]
For $t>0$ let $\varphi_t(x)=t^{-N}\varphi(x/t)$, and denote 
\[
	f_t(x)=\int_{\RR^N}\widetilde{\delta}(x-y)\varphi_t(y)dy.
\]
Then $f_t$ is of class $C^{\infty}$ on $\RR^N$, nonnegative and $f_t\rightarrow \widetilde{\delta}$ uniformly on each 
compact subset of $\RR^N$ as $t\rightarrow0$. Take a sequence $\varepsilon_n=1/2^n$ and choose $t_n=t(\varepsilon_n)$, such that 
\[
	|f_{t_n}(x)-\widetilde{\delta}(x)|<\varepsilon_n\quad\forall x\in\overline{D}.
\]
Let 
\[
	\widetilde{\delta}_n(x)=f_{t_n}(x)+2\varepsilon_n
\]
and denote

\[
	\Omega_n=\left\{x\in D:\widetilde{\delta}_n(x)<(1+\alpha)\delta(x),\quad\delta(x)<h-\varepsilon_n \right\}.
\]
\\
Because $\widetilde{\delta}_n(x)>\widetilde{\delta}(x)+\varepsilon_n$ on $\overline{D}$, we have 
$\overline{\Omega}_n\subset\Omega(a,h,\rho,\alpha)$. Moreover,

\[
	\widetilde{\delta}_n(x)<\widetilde{\delta}(x)+3\varepsilon_n<\widetilde{\delta}_{n-1}(x)-\varepsilon_{n-1}+
	\frac{3}{2}\varepsilon_{n-1}<\widetilde{\delta}_n(x),\quad x\in\overline{D},
\]
\\
hence $\Omega_{n-1}\subset\Omega_{n}$. Since $\widetilde{\delta}_n\stackrel{n}{\rightarrow}\widetilde{\delta}$ 
on $\overline{D}$, we have 

\[
	\Omega(a,h,\rho,\alpha)=\bigcup^{\infty}_{n=1}\Omega_n,
\]
\\
and thus the properties 1-3 are proved.

Now the boundary $B_n$ of $\Omega_n$ consists of two pieces $B^{1}_{n},B^{2}_{n}$, such that 

\[
	B^{1}_{n}=\left\{x\in D:(1+\alpha)\delta(x)-\widetilde{\delta}_n(x)=0,\quad \delta(x)\leq h-\varepsilon_n\right\},
\]

\[
	B^{2}_{n}=\left\{x\in D:(1+\alpha)\delta(x)-\widetilde{\delta}_n(x)\geq0,\quad\delta(x)=h-\varepsilon_n\right\}.
\]
\\
It is easily seen, similarly as in the case of $\delta(x)$, that the function $\widetilde{\delta}(x)$ satisfies 
the Lipschitz condition 
\[
	|\widetilde{\delta}(x)-\widetilde{\delta}(y)|\leq |x-y|,
\]
\\
since $E(a,\rho)$ is compact. Therefore

\[
	|\widetilde{\delta}_n(x)-\widetilde{\delta}_n(y)|=
	\left|\int_{\RR^N}(\widetilde{\delta}(x-z)-\widetilde{\delta}(y-z))\varphi_{t_n}(z)dz\right|
\]

\[
	\leq\int_{\RR^N}\left|\widetilde{\delta}(x-z)-\widetilde{\delta}(y-z)\right|\varphi_{t_n}(z)dz
	\leq\int_{\RR^N}|x-y|\varphi_{t_n}(z)dz=|x-y|,
\]
\\
and thus
\[
	\left|\frac{\partial \widetilde{\delta}_n}{\partial x_i}(x)\right|\leq 1,\quad i=1,...,N.
\]
\\
Because $h$ satisfies the condition of Corollary 3.1, for $x\in\Omega(a,h,\rho,\alpha)$ we have 

\[
	\left|\frac{\partial\delta}{\partial x_i}(x)\right|\leq 1,\quad i=1,...,N.
\]
\\
Moreover, by Lemma 3.5 we conclude

\[
	\frac{\partial\delta}{\partial x_N}(x)=\langle e_N,-\nu_{\pi(x)}\rangle	=\langle -\nu_a,-\nu_{\pi(x)}\rangle
	=\frac{2-|\nu_a-\nu_{\pi(x)}|^2}{2}
\]

\[
	\geq 1-\frac{c^{2}_{0}}{2}|a-\pi(x)|^2\geq 1-\frac{c^{2}_{0}}{2}(|a-x|+\delta(x))^2 
	\geq 1-\frac{c^{2}_{0}}{2}(\widetilde\delta(x)(x)+\rho+h)^2
\]

\[
	\geq 1-\frac{c^{2}_{0}}{2}((2+\alpha)h+\rho)^2\geq 1-\frac{\alpha}{2(1+\alpha)}.
\]
\\
Therefore

\[
	\frac{\partial}{\partial x_N}\left((1+\alpha)\delta(x)-\widetilde{\delta}_n(x)\right)=
	(1+\alpha)\frac{\partial\delta}{\partial x_N}(x)-\frac{\partial \widetilde{\delta}_n}{\partial x_N}(x)
\]

\[
	\geq (1+\alpha)\left(1-\frac{\alpha}{2(1+\alpha)}\right)-1=\frac{\alpha}{2},\quad\forall x\in\Omega(a,h,\rho,\alpha).
\]
\\
Hence, if we denote $F_n(x)=(1+\alpha)\delta(x)-\widetilde{\delta}_n(x)$, then for every $x\in B^{1}_{n}$, $F_n(x)=0$ and 
\[
	\frac{\partial F_n}{\partial x_N}(x)\geq\frac{\alpha}{2},\quad
	\left|\frac{\partial F_n}{\partial x_i}(x)\right|\leq\alpha+2,\quad\forall x\in\Omega(a,h,\rho,\alpha),i=1,...,N.
\]
\\
Denote $x=(\overline{x},x_N)$, where $\overline{x}\in\RR^{N-1}$ and $x_N\in\RR$. 
By implicit function theorem, for every $x\in B^{1}_{N}$ there exist $r_x>0$, balls $B_N(x,r_x)\subset\RR^N$, $B_{N-1}(\overline{x},r_x)\subset\RR^{N-1}$ and a function $g_x\colon B_{N-1}(\overline{x},r_x)\to\RR$ of class $C^1$, such that:
\begin{enumerate}
	\item $B_N(x,r_x)\subset\Omega(a,h,\rho,\alpha)$
	\item $\left\{(\overline{y},g_x(\overline{y})):\overline{y}\in B_{N-1}(\overline{x},r_x)\right\}\subset\Omega(a,h,\rho,\alpha)$
	\item $B^{1}_{n}\cap B_N(x,r_x)\subset\left\{(\overline{y},g_x(\overline{y})):\overline{y}\in B_{N-1}(\overline{x},r_x)\right\}$
	\item for every $\overline{y}\in B_{N-1}(\overline{x},r_x)$, $F_n(\overline{y},g_x(\overline{y}))=0$ and
				\[
					\frac{\partial g_x}{\partial x_i}(\overline{y})=-\frac{\frac{\partial F_n}{\partial x_i}(\overline{y},g_x(\overline{y}))}
					{\frac{\partial F_n}{\partial x_N}(\overline{y},g_x(\overline{y}))}.
				\]
\end{enumerate}
Since $B^{1}_{n}$ is compact, we may choose $r>0$, such that for every $x\in B^{1}_{n}$, the conditions 1-4 holds with $r_x=r$. 

Denote $V_n=\left\{\overline{x}:x\in B^{1}_{n}\right\}$, $V=\left\{\overline{x}:x\in\Omega(a,h,\rho,\alpha)\right\}$. 
Then $V_n$ is compact, and because of 2, we conclude
\[
	V_n\subset \bigcup_{x\in V_n}B_{N-1}(\overline{x},r/3)\subset\Omega(a,h,\rho,\alpha).
\]
\\
By Lemma 5.1, we can choose a finite family of disjoint balls $\left\{B_{N-1}(\overline{x}_j,r/3)\right\}$, such that
\[
	V_n\subset \bigcup_jB_{N-1}\left(\overline{x}_j,r\right).
\]
\\
Now observe, that if $x,x'\in\Omega(a,h,\rho,\alpha)$ and $\overline{x}=\overline{x}'$, then, by (small modification of) Lemmas 6.4 and 6.5, for every $\theta\in(0,1)$, $\theta x+(1-\theta)x'\in\Omega(a,h,\rho,\alpha)$. Since 
\[
	\frac{\partial F_n}{\partial x_N}(x)\geq\frac{\alpha}{2},\quad\forall x\in\Omega(a,h,\rho,\alpha),
\]
the projection $\Omega(a,h,\rho,\alpha)\cap\left\{x:F_n(x)=0\right\}\ni x\mapsto\overline{x}\in V$ 
is "one to one", and we have
\[
	B^{1}_{n}\subset\bigcup_{j\in J}\left\{(\overline{y},g_{x_j}(\overline{y})):\overline{y}\in B_{N-1}(\overline{x}_j,r)\right\}.
\]
Therefore
\[
	\sigma_n(B^{1}_{n}) \leq\sum_{j\in J}\int_{B_{N-1}(\overline{x}_j,r)}\sqrt{1+\sum^{N-1}_{i=1}
	\left(\frac{\partial g_{x_j}}{\partial x_i}(\overline{y})\right)^2}d\overline{y}
\]

\[
	\leq\sqrt{1+(N-1)\left(\frac{2(\alpha+2)}{\alpha}\right)^2}\sum_{j\in J}m_{N-1}(B_{N-1}(\overline{x}_j,r))
\]

\[
	=c_{\alpha}3^{N-1}\sum_{j\in J}m_{N-1}(B_{N-1}(\overline{x}_j,r/3))=c_{\alpha}3^{N-1}
	m_{N-1}\left(\bigcup_{j\in J}B_{N-1}(\overline{x}_j,r/3)\right)
\]

\[
	\leq c_{\alpha}3^{N-1}m_{N-1}(V)\leq c_{\alpha}3^{N-1}m_{N-1}(\left\{\overline{x}:x\in D\right\})<\infty,
\]
\\
where $m_{N-1}$ denotes $N-1$ dimensional Lebesque measure. Hence $\sigma_n(B^{1}_{n})$ is bounded by a positive constant, which does not depend on $n$. 

Now $B^{2}_{n}$ is a portion of the set $\left\{x\in\RR^N:\delta(x)=h-\varepsilon_n\right\}$, moreover 
\[
	|\nabla\delta(x)|\leq \sqrt{N}
\]
and
\[
	\left|\frac{\partial\delta}{\partial x_N}(x)\right|\geq 1-\frac{\alpha}{2(1+\alpha)}
\]
for every $x\in B^{2}_{N}$. Hence, by very similar arguments, $\sigma_n(B^{2}_{n})$ 
is uniformly bounded with respect to $n$.
\\
\end{proof}

\begin{proof4}
By the use of Lemma 6.8, we may assume, without loss of generality (neglecting a set of arbitrarily small measure), 
that $E$ is closed and $u$ is bounded in the region 
\[
	\Omega = \bigcup_{y\in E}\Gamma_{\alpha}(y) 
\] 
for every $\alpha>0$. We may also assume, that $u$ is real valued. 

Choose $\alpha>0$ and let $h,\rho$ be the constants from Lemma 7.1 with respect 
to $\alpha$; assume additionally, that $h$ satisfies the condition of Corollary 3.1. Fix $a\in E$; we shall prove that 
\[
	S^{h}_{\alpha}u(y)=\int_{\Gamma^{h}_{\alpha}(y)}|\nabla u(x)|^{2}(\delta(x))^{2-N}dx
\]
is finite for almost every $y\in E(a,\rho)=E\cap \overline{B}(a,\rho)$. It suffices to show that
\[
	\int_{E(a,\rho)}S^{h}_{\alpha}u(y)d\sigma(y)<\infty.
\]

Let $\chi(x,y,\alpha)$ be the characteristic function of $\Gamma^{h}_{\alpha}(y)$. That is, 
$\chi(x,y,\alpha)=1$ if $|x-y|<(1+\alpha)\delta(x)$ and $\delta(x)<h$, otherwise $\chi(x,y,\alpha)=0$.

We have

\[
	\int_{E(a,\rho)}S^{h}_{\alpha}u(y)d\sigma(y)=
	\int_{E(a,\rho)}\left(\int_{\Gamma^{h}_{\alpha}(y)}|\nabla u(x)|^{2}(\delta(x))^{2-N}dx\right)d\sigma(y)
\]

\[
	=\int_{E(a,\rho)}\left(\int_{\Omega(a,h,\rho,\alpha)}\chi(x,y,\alpha)|\nabla u(x)|^{2}(\delta(x))^{2-N}dx\right)d\sigma(y)
\]

\[
	=\int_{\Omega(a,h,\rho,\alpha)}\left(\int_{E(a,\rho)}\chi(x,y,\alpha)d\sigma(y)\right)
	|\nabla u(x)|^{2}(\delta(x))^{2-N}dx,
\]
\\
where 
\[
	\Omega(a,h,\rho,\alpha) = \bigcup_{y\in E(a,\rho)}\Gamma^{h}_{\alpha}(y).
\]
\\
However

\[
	\int_{E(a,\rho)}\chi(x,y,\alpha)d\sigma(y)\leq \int_{K(\pi(x),(2+\alpha)\delta(x))}d\sigma(y)=
	\sigma\left\{K(\pi(x),(2+\alpha)\delta(x))\right\},
\]
\\
and by Lemma 3.3, there exists a positive constant $c_{\alpha}$, such that

\[
	\sigma\left\{K(\pi(x),(2+\alpha)\delta(x))\right\}\leq c_{\alpha}\delta(x)^{N-1}.
\]
\\
Thus it suffices to show that 

\[
	\int_{\Omega(a,h,\rho,\alpha)}\delta(x)|\nabla u(x)|^{2}dx<\infty.
\]
\\
We shall transform the last integral by Green's theorem. In order to do this we shall use the approximating smooth regions $\Omega_n$ discussed in Lemma 7.2. Therefore, by the properties of $\Omega_n$, the last inequality is 
equivalent with

\[
	\int_{\Omega_n}\delta(x)|\nabla u(x)|^{2}dx\leq c<\infty,
\]
\\
where the constant $c$ is independent of $n$. Since the region $\Omega_n$ has a sufficiently smooth boundary $B_n$, we apply to it Green's theorem in the form

\[
	\int_{B_n}\left(G\frac{\partial F}{\partial\nu_n}-F\frac{\partial G}{\partial \nu_n}\right)d\sigma_n=
	\int_{\Omega_n}\left(G\Delta F-F\Delta G\right)dx.
\]
\\
Here $\partial/\partial\nu_n$ indicates the directional derivative along the outward normal to $B_n$. 

In the above formula we take $F=u^2$, and $G=\delta$. A simple calculation shows that 
$\Delta\left(u^2\right)=2|\nabla u|^2$, since $u$ is real valued and harmonic. Therefore we obtain 

\[
	2\int_{\Omega_n}\delta(x)|\nabla u(x)|^2dx=\int_{\Omega_n}u^2(x)\Delta\delta(x)dx
\]

\[
	+\int_{B_n}\left(\delta(x)\frac{\partial u^2}{\partial\nu_n}(x)-
	u^2(x)\frac{\partial \delta}{\partial\nu_n}(x)\right)d\sigma_n(x).
\]
\\

Now because $|\delta(x)-\delta(y)|\leq|x-y|$, we conclude, as in the proof of Lemma 7.2, that for $x\in\Omega(a,h,\rho,\alpha)$

\[
	\left|\frac{\partial\delta}{\partial x_k}(x)\right|\leq1,\quad k=1,...,N,
\]
\\
since $h$ satisfies the condition of Corollary 3.1. Because $\overline{\Omega}_n\subset\Omega(a,h,\rho,\alpha)$, the inequality holds for $x\in\overline{\Omega}_n$. Moreover, by Lemma 3.5, there exists a constant $M>0$ such 
that for every $x,y\in\Omega(a,h,\rho,\alpha)$ we have

\[
	\left|\frac{\partial\delta}{\partial x_k}(x)-\frac{\partial\delta}{\partial x_k}(y)\right|\leq M|x-y|,\quad k=1,...,N.
\]
\\
Therefore 

\[
	\left|\frac{\partial^2\delta}{\partial x^{2}_{k}}(x)\right|\leq M\quad\forall x\in\overline{\Omega}_n, 
	\quad k=1,...,N.
\]
\\
For $\beta>\alpha$ we have 

\[
	\overline{\Omega}_n\subset\Omega(a,h,\rho,\alpha)\subset\bigcup_{y\in E}\Gamma_{\beta}(y).
\]
\\
Hence, there exists $c_1=c_1(\beta)>0$, such that $|u|\leq c_1$ on $\overline{\Omega}_n$. Notice also that 
$\partial u^2/\partial\nu_n=2u\cdot\partial u/\partial\nu_n$. Thus

\[
	\left|\delta(x)\frac{\partial u^2}{\partial\nu_n}(x)\right|\leq 2|u(x)|\cdot\delta(x)
	\cdot\left|\frac{\partial u}{\partial\nu_n}(x)\right|\leq 2|u(x)|\cdot\delta(x)\cdot|\nabla u(x)|
\]
\\
for $x\in\overline{\Omega}_n$. By Lemma 7.1, there exists $c_2=c_2(\alpha,\beta)$, such that $\delta(x)|\nabla u(x)|\leq c_2$ on $\Omega(a,h,\rho,\alpha)$. Therefore we have

\[
	\left|\int_{\Omega_n}u^2(x)\Delta\delta(x)dx
	+\int_{B_n}\left(\delta(x)\frac{\partial u^2}{\partial\nu_n}(x)-
	u^2(x)\frac{\partial\delta}{\partial\nu_n}(x)\right)d\sigma_n(x)\right|
\]

\[
	\leq\int_{\Omega_n}|u(x)|^2|\Delta\delta(x)|dx
	+\int_{B_n}\left(\left|\delta(x)\frac{\partial u^2}{\partial\nu_n}(x)\right|
	+|u(x)|^2\left|\frac{\partial \delta}{\partial\nu_n}(x)\right|\right)d\sigma_n(x)
\]

\[
	\leq c^{2}_{1}\cdot NM\cdot m_N(\Omega_n)+2c_1c_2\sigma_n(B_n)
	+c^{2}_{1}\int_{B_n}|\nabla\delta(x)|d\sigma_n(x)
\]

\[
	\leq c^{2}_{1}\cdot NM\cdot m_N(D)+\left(2c_1c_2+Nc^{2}_{1}\right)\sigma_n(B_n),
\]
\\
where $m_N$ is $N$ dimensional Lebesque measure. By Lemma 7.2, $\sigma_n(B_n)$ is uniformly bounded, so the proof is complete.
\\
\end{proof4}

\bigskip


\markright{References}


\begin{thebibliography}{XXXX}


\bibitem{ABR} {\sc S. Axler, P. Bourdon, and W. Ramey},
              {\it Harmonic Function Theory},
               Springer-Verlag, New York, Berlin, \textbf{1992}

\bibitem{AKSZ} {\sc H. Aikawa, T. Kilpel{\"a}inen, N. Shanmugalingam, and X. Zhong},
               \emph{Boundary Harnack Principle for p-harmonic Functions in Smooth Euclidean Domains}, 
               Potential Analysis 26 (2006), 281-301.

\bibitem{K}   {\sc S.G. Krantz},
              {\it Function theory of several complex variables},
               John Wiley \& Sons, New York, Chichester, Brisbane, Toronto, Singapore \textbf{1982}

\bibitem{SP}  {\sc M. Spivak}, 
              \emph{Calculus on Manifolds}, Benjamin, New York, \textbf{1965}.

\bibitem{S1}  {\sc E.M. Stein}, 
              \emph{Boundary behavior of holomorphic functions of several complex variables}, 
              Princeton University Press and University of Tokyo Press, Princeton, New Jersey, \textbf{1972}.

\bibitem{S2}  {\sc E.M. Stein}, 
              \emph{On the theory of harmonic functions of several variables. II.Behavior near the boundary}, 
              Acta Math. 106 (1961), 137-174.

\bibitem{W}   {\sc K.O. Widman}, 
              \emph{On the boundary behavior of solutions to a class of elliptic partial differential equations}, 
              Ark. Mat. 6 (1966), 485-533.


\end{thebibliography}
\end{document}